\documentclass[11pt]{article}

\usepackage{paralist, amsmath, amsthm, amssymb, color}
\usepackage[margin=1in]{geometry}
\usepackage{subfigure}

\usepackage{tikz}
\tikzstyle{bv}=[circle,draw=black!90,fill=black!100,thick,inner
sep=2pt,minimum width=5pt]
\usepackage{hyperref}

\usepackage[boxsize=0.5em]{ytableau}

\definecolor{DarkBlue}{rgb}{0,0.08,0.45} 
\definecolor{DarkGreen}{rgb}{0,0.5,0.0}

\newtheorem*{theorem*}{Theorem}
\newtheorem*{corollary*}{Corollary \ref{skewcor}}
\newtheorem*{theorem3*}{Theorem \ref{thm:indecomp}}
\newtheorem*{recapprop*}{Proposition \ref{RL:qprop}}
\newtheorem*{conjecture*}{Conjecture \ref{conj:rothe}}
\numberwithin{equation}{section}
\newtheorem{theorem}[equation]{Theorem}
\newtheorem{proposition}[equation]{Proposition}

\newtheorem{corollary}[equation]{Corollary}
\newtheorem{conjecture}[equation]{Conjecture}
\newtheorem{question}[equation]{Question}

\theoremstyle{definition}

\newtheorem{example}[equation]{Example}	
\newtheorem{examples}[equation]{Examples}
\newtheorem{rmk}[equation]{Remark}
\newtheorem{rmks}[equation]{Remarks}
\newenvironment{remark}{\begin{rmk}}{\hfill $\blacksquare$ \end{rmk}}

\newcommand{\F}{\mathbf{F}_q}

\newcommand{\ZZ}{\mathbf{Z}}
\newcommand{\id}{\textrm{id}}

\DeclareMathOperator{\inv}{inv}
\DeclareMathOperator{\mat}{mat}
\DeclareMathOperator{\rk}{rank}

\title{Counting matrices over finite fields with support on skew Young diagrams and complements of Rothe diagrams} 
\author{Aaron J. Klein (Brookline High School)\\
Joel Brewster Lewis (University of Minnesota)%
\thanks{Supported by NSF RTG grant NSF/DMS-1148634}
\\
Alejandro H. Morales (LaCIM-Universit\'e du Qu\'ebec \`a Montr\'eal)%
\thanks{Supported by a CRM-ISM Postdoctoral Fellowship}
}
\date{\today}

\begin{document}
\setcounter{tocdepth}{2}

\maketitle

\begin{abstract}
We consider the problem of finding the number of matrices over a finite field with a certain rank and with support that avoids a subset of the entries. These matrices are a $q$-analogue of permutations with restricted positions (i.e., rook placements). For general sets of entries these numbers of matrices are not polynomials in $q$ (Stembridge 98); however, when the set of entries is a Young diagram, the numbers, up to a power of $q-1$, are polynomials with nonnegative coefficients (Haglund 98). 

In this paper, we give a number of conditions under which these numbers are polynomials in $q$, or even polynomials with nonnegative integer coefficients.  We extend Haglund's result to complements of skew Young diagrams, and we apply this result to the case when the set of entries is the Rothe diagram of a permutation. In particular, we give a necessary and sufficient condition on the permutation for its Rothe diagram to be the complement of a skew Young diagram up to rearrangement of rows and columns. We end by giving conjectures connecting invertible matrices whose support avoids a Rothe diagram and Poincar\'e polynomials of the strong Bruhat order.
\end{abstract}


\section{Introduction}

We study certain $q$-analogues of permutations with restricted positions, or equivalently of placements of non-attacking rooks.  The $q$-analogue of permutations we work with is invertible $n\times n$ matrices over the finite field $\F$ with $q$ elements, as in \cite[Ch.\ 1]{EC1}. Then the analogue of permutations with restricted positions is invertible matrices over $\F$ with some entries required to be zero.

Specifically, given a subset $S$  of $\{1,2,\ldots,n\} \times \{1,2,\ldots,n\}$, let $\mat_q(n,S,r)$ be the number of $n\times n$ matrices over $\F$ with rank $r$, none of whose nonzero entries lie in $S$.  This is clearly an analogue (in the plain English meaning) of the problem of counting permutations whose permutation matrix has no $1$ in the position of any entry of $S$, but actually much more can be said.  In \cite[Prop.\ 5.1]{LLMPSZ} it was shown that $\mat_q(n,S,r)/(q-1)^r$ is in fact an \emph{enumerative} $q$-analogue of permutations with restricted positions; that is, its value, modulo $(q-1)$, counts the placements of $r$ non-attacking rooks on the complement of $S$. 

The function $\mat_q(n,S,r)$ can exhibit a variety of different behaviors, as seen in the following three examples.
\begin{examples}
\begin{compactenum}      
\item When $S = \varnothing$, $\mat_q(n,\varnothing,n)$ is the number of $n\times n$ invertible matrices over $\F$, which is 
\[
 (q^{n}-1)(q^{n}-q)\cdots (q^n-q^{n-1})=q^{\binom{n}{2}}(q-1)^n \prod_{i=1}^{n} (1+q+\cdots+q^{i-1}).
\]
The term $\prod_{i=1}^n(1+q+\cdots+q^{i-1})$ in the product is a polynomial with positive coefficients, and in fact is the generating series for permutations in $\mathfrak{S}_n$ by number of inversions.

\item  When $n=3$ and $S$ is the diagonal $\{(1,1),(2,2),(3,3)\}$ we have
\[\mat_q(3,\{(1,1),(2,2),(3,3)\},3)=(q-1)^3(q^3+2q^2-q).\]
The number of invertible $n\times n$ matrices for general $n$ over $\F$ with zero diagonal was computed in \cite[Prop. 2.2]{LLMPSZ}; as in this example, it is of the form $(q - 1)^n f(q)$ for a polynomial $f$ with both positive and negative  coefficients.


\item When $n=7$, Stembridge \cite{Stem} found a set $F$ with $28$ elements (shown in Figure~\ref{fig:Fano}) such that $\mat_q(7,F,7)$ is given by a quasi-polynomial in $q$, that is, by two distinct polynomials depending on whether $q$ is even or odd. The set $F$ is the complement of the incidence matrix of the {\em Fano plane}.
\end{compactenum}
\end{examples}

\begin{figure}\begin{center}
\raisebox{.75in}{$
\left[\begin{array}{ccccccc}
a_{{1\textcolor{red}{1}}}&a_{{12}}&0&0&0&0&a_{{17}}
\\ a_{{2\textcolor{red}{1}}}&0&a_{{23}}&0&0&a_{{26}}&0
\\ a_{{3\textcolor{red}{1}}}&0&0&a_{{34}}&a_{{35}}&0&0
\\ 0&a_{{42}}&a_{{43}}&0&a_{{45}}&0&0
\\ 0&a_{{52}}&0&a_{{54}}&0&a_{{56}}&0
\\ 0&0&a_{{63}}&a_{{64}}&0&0&a_{{67}}
\\ 0&0&0&0&a_{{75}}&a_{{76}}&a_{{77}}\end{array}\right]
$}
\hspace{1in}
\begin{tikzpicture}[scale=0.75,thick]
                \node [style=bv] (1) at (0, 3.464) [label=above:$1$] {};
                \node [style=bv] (3) at (1, 1.732) [label=above:$4$] {}; 
                \node [style=bv] (4) at (0, 1.155) [label=left:$7$] {};
                \node [style=bv] (5) at (-2, 0) [label=below:$3$] {};
                \node [style=bv] (6) at (0, 0) [label=below:$6$] {};
                \node [style=bv] (7) at (2, 0) [label=below:$5$]  {};
                \draw (0,1.155) circle (1.155cm);
                \draw[red] [-] (1) -- (5);
                \node [style=bv] (2) at (-1, 1.732) [label=above:$2$] {};
                \draw [-] (1) -- (7); 
                \draw [-] (1) -- (6);
                \draw [-] (7) -- (5);
                \draw [-] (5) -- (3);
                \draw [-] (7) -- (2);
\end{tikzpicture}
\end{center}
\caption{A representative matrix counted in $\mat_q(7,F,7)$ where $F$ is the complement of the point-line incidence matrix of the Fano plane, shown at right. Stembridge \cite{Stem} showed this to be the smallest example of the form $\mat_q(n, S, n)$ that is not a polynomial in $q$.}
\label{fig:Fano}
\end{figure}
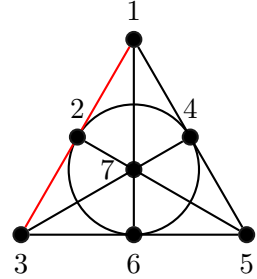

From the examples above we see that $\mat_q(n,S,r)$ is not necessarily a polynomial in $q$, and if it is a polynomial in $q$ it might or might not be of the form $(q-1)^rf(q)$ where $f(q)$ is a polynomial with nonnegative integer coefficients.  Then a natural question to ask is the following:

\begin{question} 
What families of sets $S$ are there such that $\mat_q(n,S,r)/(q-1)^r$ is (i) not a polynomial in $q$, (ii) a polynomial in $q$, or (iii) a polynomial in $q$ with nonnegative integer coefficients? 
\end{question}

In the remainder of this introduction, we give a summary of our progress towards answering this question.

\subsection*{Outline and summary of results}

In Section~\ref{sec:def}, we give the definitions and notation that will be used throughout the paper including the definition and some properties of $q$-rook numbers.  

In Section~\ref{sec:general conditions}, we address general conditions on $r$ and $S$ under which the function $\mat_q(n,S,r)$ is always a polynomial in $q$.  We show that if $r = 1$ then $\mat_q(n, S, 1)$ is a polynomial in $q$ for any set $S$, though not necessarily with nonnegative coefficients.  (It is an open question whether there is a set $S$ such that $\mat_q(n, S, 2)$ is non-polynomial in $q$.)  Our main result of this section is to extend work of Stembridge \cite{Stem} to give reductions to compute $\mat_q(n, S, r)$ in terms of smaller instances of similar problems when $S$ has a row or column with very few or very many entries.

In the rest of the paper, we discuss special families of sets $S$ such that $\mat_q(n,S,r)/(q-1)^r$ is a polynomial in $q$ with nonnegative integer coefficients. Haglund \cite{Hag} showed that if the set $S$ is a {\em straight shape} then $\mat_q(n,\overline{S},r)/(q-1)^r$ is a polynomial with nonnegative integer coefficients. Our second main result, proved in Section~\ref{sec:NE}, is to extend this to complements of {\em skew shapes}.

\begin{corollary*} For any skew shape $S_{\lambda/\mu}$,
\[\mat_q(n,\overline{S_{\lambda/\mu}},r) = (q-1)^r f(q),\]
where $f(q)$ is a polynomial with nonnegative integer coefficients.
\end{corollary*}
In fact, we show that this is true for an even larger class of shapes than skew shapes, namely those that have what we call the {\em North-East Property}.  Also, because $\mat_q(n,S,r)$  is invariant under permuting rows and columns we have that $\mat_q(n,\overline{S},r)/(q-1)^r$ is a polynomial with nonnegative integer coefficients for any set $S$ that is a straight or skew shape after permuting rows and columns. 

Another natural family of diagrams is the collection of {\bf Rothe diagrams} of permutations, which appear in the study of Schubert calculus. The Rothe diagram $R_w$ of a permutation $w$ is a subset of $\{1,2,\ldots,n\} \times \{1,2,\ldots,n\}$ whose cardinality is equal to the number of inversions of $w$; it is given by 
\[
R_w = \{(i,j) \mid 1\leq i, j\leq n,\,\, w(i)>j,\,\, w^{-1}(j)>i\}.
\] 
See Figure~\ref{fig:exRothe} for some examples of Rothe diagrams. Lascoux and Sch\"{u}tzenberger showed in \cite{LS} that the Rothe diagram $R_w$ of a permutation $w$ is a straight shape up to permutation of rows and columns if and only if $w$, written as a word $w_1w_2\cdots w_n$, avoids the permutation pattern $2143$ (i.e., there is no sequence $i<j<k<l$ such that $w_j<w_i<w_l<w_k$). Our third main result is to give an analogous criterion for the case of complements of skew shapes.

\begin{theorem3*} 
The Rothe diagram $R_w$ of a permutation $w$ is, up to permuting its rows and columns, the complement of a skew shape if and only if $w$ can be decomposed as $w = a_1a_2\ldots a_kb_1b_2\ldots b_{n-k}$ where $a_i<b_j$ for all $i$ and $j$, and both $a_1a_2\ldots a_k$ and $b_1b_2\ldots b_{n-k}$ are $2143$-avoiding.
\end{theorem3*}

We also show that this condition is equivalent to the statement that $w$ avoids the nine patterns  $24153$, $25143$, $31524$, $31542$, $32514$, $32541$, $42153$, $52143$, and $214365$, and we express the generating series for these permutations in terms of the generating series for $2143$-avoiding permutations.

By Corollary~\ref{skewcor}, if $w$ satisfies the condition above then $\mat_q(n,R_w,r)/(q-1)^r$ is a polynomial with nonnegative integer coefficients.
Surprisingly, computer calculations for $n\leq 7$ and $0\leq r\leq n$ \cite{Code} suggest that $\mat_q(n,R_w,r)/(q-1)^r$ is a polynomial with nonnegative integer coefficients for \emph{all} permutations $w$ (see Conjecture~\ref{conj:rothe}). Moreover, computer calculations also suggest that when  $w$ avoids the permutation patterns $1324, 24153, 31524$, and $426153$ we have that $\mat_q(n,R_w,n)/(q-1)^n$ is (up to a power of $q$) the {\em Poincar\'e polynomial} $P_w(q) = \sum_{u\geq w} q^{\inv(u)}$, where the sum is over all permutations $u$ of $n$ above $w$ in the {\em strong Bruhat order} (see Conjecture~\ref{conj:PoinRothe}). Interestingly, these four patterns have appeared in related contexts \cite{GR2, AP, Sj, HLSS}.

Supplementary code for calculating $\mat_q(n,S,r)$ and other related objects and data generated by this code to test the conjectures in Section~\ref{sec:poincare} are available at the following website:
\begin{center}
\url{http://sites.google.com/site/matrixfinitefields/}
\end{center}

\subsection*{Acknowledgements} 
We would like to express thanks to the following individuals for their contributions: Alexander Postnikov, for pointing us in the direction of Rothe diagrams; Richard Stanley and Steven V Sam, for helpful suggestions and discussions; John Stembridge, for making his Maple {\tt reduce}  package \cite{reduce} publicly available;  Sara Billey and Brendan Pawlowski, for mentioning the example in Remark~\ref{rem:stansymmfcn}; and the anonymous referee for several helpful comments including the observation at the end of Remark~\ref{rem:rank1t}.  AJK is grateful to the PRIMES program at the MIT Mathematics Department, where this research was done.  JBL and AHM would like to dedicate this paper to the memories of Richard Geller  and Nicol\'as del Castillo.

\section{Definitions} \label{sec:def}

We denote $[n]=\{1,2,\ldots,n\}$. The {\bf support} of a matrix $A$ is the set of indices $(i,j)$ of the nonzero entries $a_{ij}\neq 0$.  Fix integers $n$ and $r$ such that $n \geq 1$ and $n \geq r \geq 0$, and let $S$ be a subset of $[n]\times [n]$.  We define $\mat_q(n,S,r)$ to be the number of $n \times n$ matrices over $\F$ with rank $r$ and support contained in $\overline{S}$, the complement of $S$.  That is, $\mat_q(n, S, r)$ counts matrices $A$ of rank $r$ such that if $(i,j)\in S$ then $a_{ij} = 0$.  We consider the problem of computing $\mat_q(n,S,r)$.  

We now define several special types of diagrams that will be important to us in what follows.  Examples of these diagrams are given in Figure~\ref{fig:strtskshapes}.  We say that $S \subseteq [n]\times [n]$ is a {\bf straight shape} if its elements form a Young diagram of a partition. (Throughout this paper we use {\em English notation} and matrix coordinates for partitions.) Thus, to every integer partition $\lambda$ with at most $n$ parts and with largest part at most $n$ (i.e., to each sequence of integers $(\lambda_1,\lambda_2,\ldots,\lambda_n)$ such that $n \geq \lambda_1\geq \lambda_2 \geq \cdots \geq \lambda_n \geq 0$) there is an associated set $S = S_\lambda=\{(i,j) \mid 1\leq j\leq \lambda_i\}$. We denote by $|\lambda|$ the {\bf size} $\lambda_1+\lambda_2+\cdots$ of the shape $\lambda$. This is also the number of entries in $S_{\lambda}$. Similarly, if $\lambda$ and $\mu$ are partitions such that $S_{\mu}\subseteq S_{\lambda}$ then we say that the set $S_{\lambda}\setminus S_{\mu}$ is a {\bf skew shape} and we denote it by $S_{\lambda/\mu}$.  Lastly, we say that a set $S\subseteq [n]\times [n]$ has the {\bf North-East ({\sf NE}) Property} if for all $i, i', j, j' \in [n]$ such that $i'<i$ and $j<j'$ we have that if $(i,j), (i',j)$, and $(i,j')$ are in $S$, then so is $(i',j')$. Note that for any partitions $\lambda$ and $\mu$, $S_{\lambda}, \overline{S_{\lambda}}$, and $S_{\lambda/\mu}$ have the {\sf NE} Property. But $\overline{S_{\lambda/\mu}}$ in general does not have this property.

We denote by $\mathfrak{S}_n$ the group of permutations on $[n]$. We write permutations as words $w=w_1w_2\cdots w_n$ where $w_i$ is the image of $w$ at $i$. Let $\inv(w)$ denote the number of {\bf inversions} $\#\{(i,j) \mid i<j, w_i>w_j\}$ of $w$. We also identify each permutation $w$ with its {\bf permutation matrix}, the $n\times n$ $0$-$1$ matrix with $1$s in positions $(i,w_i)$. 

We think of the $1$s in a permutation matrix as $n$ non-attacking rooks on $[n]\times [n]$. In this case, the number of inversions of the permutation is exactly the number of elements in $[n]\times [n]$ that are not directly below/south (in the same column) or to the right/east (in the same row) of any placed rook. We generalize this as follows. Given a subset $B$ of $[n]\times [n]$ (sometimes called a \emph{board}) and a rook placement $C$ of $r$ non-attacking rooks on $B$, the {\bf SE-inversion number} $\inv_{{\sf SE}}(C,B)$ is the number of elements in $B$ {\em not} directly south (in the same column) or to the  east (in the same row) of any placed rook. Then the $r$th  ({\sf SE}) {\bf $q$-rook number} of Garsia and Remmel \cite{GR} is 
\[
R^{\sf (SE)}_r(B,q) = \sum_{C} q^{\inv_{{\sf SE}}(C,B)},
\]
where the sum is over all rook placements $C$ of $r$ non-attacking rooks on $B$. We define  the north east inversion number $\inv_{\sf NE}(C,B)$ and rook polynomial $R^{\sf (NE)}_r(B,q)$ analogously.

\begin{proposition}[\cite{GR}]
Given an integer partition $\lambda=(\lambda_1,\lambda_2,\ldots,\lambda_n)$ such that $n\geq \lambda_1\geq \lambda_2 \geq \cdots \geq \lambda_n\geq 0$, set $S_{\lambda} = \{(i,j) \mid 1\leq i \leq n, 1\leq i \leq \lambda_j\}$.  The Garsia-Remmel $q$-rook number $R_n^{\sf (SE)}(S_{\lambda},q)$ is
\begin{equation}\label{GRhooknum}
R_n^{\sf (SE)}(S_{\lambda},q) = \prod_{i=1}^n [\lambda_{n-i+1}-i+1]_q,
\end{equation}
where $[m]_q=1+q+q^2+\cdots +q^{m-1}$.
\end{proposition}

\begin{remark}\label{SENE}
We will see as a corollary of Theorems~\ref{jh} and~\ref{NEthm} that for a straight shape $\lambda$, $R_r^{\sf (SE)}(S_{\lambda},q)= R_r^{\sf (NE)}(S_{\lambda},q)$. However, this is not true for all skew shapes. For example, if $\lambda/\mu = 4432/31$, we have $R_3^{\sf (SE)}(S_{4432/31},q)=1+6q^2+5q^3+3q^4+2q^5+q^6$ and $R_3^{\sf (NE)}(S_{4432/31},q)=2q+8q^2 + 7q^3+q^4$. But for skew shapes in the case of $n$ rooks we do have an analogous relation, as the following result shows.
\end{remark}

\begin{proposition} \label{prop:nrksskew}
For a skew shape $S_{\lambda/\mu} \subseteq [n] \times [n]$ we have
\[
R_n^{\sf (SE)}(S_{\lambda/\mu},q) =q^{\binom{n}{2}-|\mu|}\cdot R_n^{\sf(NE)}(S_{\lambda/\mu},q^{-1}).
\]
\end{proposition}

\begin{proof}

For each rook placement of $n$ rooks on $S_{\lambda/\mu}$, the number of SE-inversions is equal to the number of inversions of the associated permutation $w$ minus the size of $\mu$. On the other hand, the number of NE-inversions of this rook placement on $S_{\lambda/\mu}$ is $\binom{n}{2}$ minus the number of inversions of $w$. The result follows.
\end{proof}




\begin{figure}
$$
\begin{array}{ccc}
\text{(i)} \quad  S_{(4,3,2)} & \text{(ii)} \quad   S_{(5,5,4,3,1)/(2,2,1)} & \text{(iii)} \quad  S\\[2mm]
\left[ \begin {array}{ccccc} 0&0&0&0&a_{{15}}\\ 0&0
&0&a_{{24}}&a_{{25}}\\ 0&0&a_{{33}}&a_{{34}}&a_{
{35}}\\ a_{{41}}&a_{{42}}&a_{{43}}&a_{{44}}&a_{
{45}}\\ a_{{51}}&a_{{52}}&a_{{53}}&a_{{54}}&a_{
{55}}\end {array} \right],
&
\left[ \begin {array}{ccccc} 
a_{{11}}&a_{{12}}&0&0&0\\ 
a_{{21}}&a_{{22}}&0&0&0\\ 
a_{{31}}&0&0&0&a_{{35}}\\ 
0&0&0&a_{{44}}&a_{{45}}\\ 
0&a_{{52}}&a_{{53}}&a_{{54}}&a_{{55}}
\end {array} \right],
&
\left[ \begin {array}{ccccc} a_{{11}}&a_{{12}}&a_{{13}}&0&0
\\ a_{{21}}&0&0&0&0\\ 0&0&0&0&0
\\ a_{41}&0&0&0&a_{{45}}\\ a_{51}&a_{52}&0&a_{{5
4}}&a_{{55}}\end {array} \right] 
\\
\text{(iv)}\quad \overline{S_{(4,3,2)}} &
\text{(v)} \quad \overline{S_{(5,5,4,3,1)/(2,2,1)}} & \text{(vi)} \quad \overline{S}\\[2mm]
\left[ \begin {array}{ccccc} a_{{11}}&a_{{12}}&a_{{13}}&a_{{14}}&0
\\ a_{{21}}&a_{{22}}&a_{{23}}&0&0
\\ a_{{31}}&a_{{32}}&0&0&0\\ 0&0
&0&0&0\\ 0&0&0&0&0\end {array} \right], 
&
\left[ \begin {array}{ccccc} 
0&0&a_{{13}}&a_{{14}}&a_{{15}}\\ 
0&0&a_{{23}}&a_{{24}}&a_{{25}}\\ 
0&a_{{32}}&a_{{33}}&a_{{34}}&0\\ 
a_{{41}}&a_{{42}}&a_{{43}}&0&0\\ 
a_{{51}}&0&0&0&0
\end {array} \right], &
\left[ \begin {array}{ccccc} 0&0&0&a_{{14}}&a_{{15}}
\\ 0&a_{{22}}&a_{{23}}&a_{{24}}&a_{{25}}
\\ a_{{31}}&a_{{32}}&a_{{33}}&a_{{34}}&a_{{35}}
\\ 0&a_{{42}}&a_{{43}}&a_{{44}}&0
\\ 0&0&a_{{53}}&0&0\end {array}
 \right] 
\end{array}
$$
\caption{Representative matrices from $\mat_q(5,S,r)$ when $S$ is (i) a straight shape, (ii) a skew shape, (iii) a set with the {\sf NE} Property; and their respective complements (iv),(v),(vi).}
\label{fig:strtskshapes}
\end{figure}

\section{General polynomiality results}\label{sec:general conditions}

In this section, we give some general conditions under which $\mat_q(n, S, r)$ is a polynomial.  In Subsection~\ref{sec:rank 1}, we show that for any $n$ and $S$, the function $\mat_q(n, S, 1)$ is polynomial in $q$.  In Subsection~\ref{sec:few or many zeroes}, we give reduction formulas for computing $\mat_q(n, S, r)$ in terms of smaller instances when $S$ has a row or column with either very few or very many entries.  

Throughout this section we work with rectangular matrices of any dimensions rather than just square matrices.  Thus, in this section, for integers $m$, $n$ and $r$ and a subset $S$ of $[m] \times [n]$, we denote by $\mat_q(m \times n, S, r)$ the number of $m \times n$ matrices of rank $r$ over $\F$ whose support avoids $S$.

\subsection{Polynomial formula for the rank-one case \texorpdfstring{$\mat_q(n,S,1)$}{mat(q;n,S,1)}} \label{sec:rank 1}

In Figure~\ref{fig:Fano} we showed an example by Stembridge~\cite{Stem} of a set $S \subseteq [7] \times [7]$ such that $\mat_q(7,S,7)$ is not a polynomial in $q$.  In this paper, we mainly focus on studying certain families of sets $S$ where $\mat_q(n,S,r)$ is a polynomial in $q$. But before looking at particular sets $S$, we consider the rank $r=1$ case for an arbitrary set $S$.

\begin{proposition}
For any $m$ and $n$ and any set $S\subseteq [m]\times [n]$, $\mat_q(m \times n,S,1)$ is a polynomial in $q$.
\end{proposition}

\begin{proof}
Fix $m$, $n$ and $S$.  We count matrices with a given collection of nonzero rows.  Given a nonempty subset $T \subseteq [m]$ of rows, let $a_S(T)$ be the number of columns with no entries which are both in one of the rows of $T$ and in $S$.  Then there are exactly $(q^{a_S(T)} - 1)(q - 1)^{\# T - 1}$ matrices of rank $1$ over $\F$ whose support avoids $S$ and whose nonzero rows are exactly those in $T$.  It follows immediately that
\begin{equation} \label{eq:rank1case}
\mat_q(m \times n, S, 1) = \sum_{\substack{T \subseteq [m] \\ \text{nonempty}}} (q^{a_S(T)} - 1)(q - 1)^{\# T - 1}
\end{equation}
is a polynomial in $q$.
\end{proof}

\begin{example} \label{negcoeffex}
Take the $4\times 4$ shape $S = \{(i,i) \mid 1 \leq i \leq 4\}$.  Then
\begin{align*}
\mat_q(4 \times 4,S,1) & = \sum_{k = 1}^4 \binom{4}{k} (q^{4 - k} - 1)(q - 1)^{k - 1} \\
& = (q - 1) \cdot 2(7q^2 - 2q + 1).
\end{align*}
(In fact one can show that if $S$ is the diagonal $\{(i,i) \mid 1 \leq i\leq n\}$ then $\mat_q(n\times n,S,1) = \frac{1}{q-1}((2q-1)^n - 2q^n+1)$.)
\end{example}

\begin{remark} \label{rem:rank1t}
In later sections of this paper, we show that for certain diagrams $S$ (straight shapes, skew shapes, and conjecturally Rothe diagrams of permutations), the function $\mat_q(n, S, r)/(q - 1)^r$ is not only a polynomial in $q$ but also has nonnegative coefficients.  However, this is not the case for matrices of rank $1$: although each summand is a power of $q - 1$ times a polynomial with positive coefficients, the powers of $q-1$ differ.  So, as in Example~\ref{negcoeffex}, negative coefficients can turn up for certain choices of $S$. Interestingly, if we substitute $t=q-1$ in \eqref{eq:rank1case} we obtain a polynomial in $t$ with nonnegative coefficients. Is this true more generally?
\end{remark}

\subsection{Reduction formulas when \texorpdfstring{$S$}{S} has rows with few or many entries}\label{sec:few or many zeroes}

In \cite[Thm.\ 8.2]{Stem}, Stembridge gave some structural restrictions on a minimal set $S$ such that $\mat_q(n, S, n)$ is non-polynomial.  In particular, he showed that in such a minimal example, every row contains at least three entries of $\overline{S}$ and at least two entries of $S$.  In this section, we push his results slightly further: we show that for any rank $r$, if either $S$ or $\overline{S}$ has a row with at most two entries then we can express $\mat_q(m \times n, S, r)$ as a linear combination with polynomial coefficients of similar expressions for matrices of strictly smaller size.  (Of course, the same arguments apply to columns as well as to rows.)   
Furthermore, we explain why this approach cannot be used in the case where $S$ contains three entries in some row.  This does not settle the question of whether there are some $m$ and $n$ and a set $S \subseteq [m] \times [n]$ with only three entries per row such that $\mat_q(m \times n, S, r)$ is non-polynomial in $q$, though we conjecture that such examples exist; in Stembridge's example \cite{Stem} (see Figure~\ref{fig:Fano}) the set $S$ has four entries per row and $\overline{S}$ has three.

For convenience, throughout this section we write expressions like $\mat_q(a \times b, S, r)$ without worrying whether $S \subseteq [a] \times [b]$, where properly we should write $\mat_q( a \times b, ([a] \times [b])\cap S, r)$.

\subsubsection{Reduction when $S$ has at most two entries in some row}

We wish to show that if $S$ has two or fewer entries in some row then $\mat_q(m \times n, S, r)$ can be reduced to a sum of polynomial multiples of similar but simpler expressions.  We begin with a useful proposition that we use in the proof of Theorem~\ref{thm:two zeroes} to do case analysis.

\begin{proposition}\label{prop:rank-nullity}
Suppose that $M = \begin{bmatrix}{\bf v}_1 & {\bf v}_2 & \cdots & {\bf v}_n\end{bmatrix}$ is an $m \times n$ matrix over $\F$ of rank $r$ and that the submatrix $\begin{bmatrix} {\bf v}_{n - k + 1} & \cdots & {\bf v}_n\end{bmatrix}$ has rank $r'$.  The number of tuples ${\bf w} = (w_1, w_2, \ldots, w_{n - k}) \in \F^{n - k}$ such that 
\[
\rk 
\begin{bmatrix}
{\bf v}_1 & \cdots & {\bf v}_{n - k} & {\bf v}_{n - k + 1} & \cdots & {\bf v}_n \\
w_1 & \cdots & w_{n - k} &       0       & \cdots &  0
\end{bmatrix}
 = r
\]
is $q^{r - r'}$.  For the other $q^{n - k} - q^{r - r'}$ tuples, this matrix has rank $r + 1$.
\end{proposition}
\begin{proof}
This is just the rank-nullity theorem: vectors $w$ that do not increase the rank are those that (when augmented by $k$ $0$s) lie in the row space of $M$.  The portion of the row space with last $k$ coordinates equal to $0$ is exactly the kernel of the projection onto these last $k$ coordinates.  The image of this projection has dimension $r'$, so the kernel has dimension $r - r'$, as desired.
\end{proof}

Now we use this to prove the main result of this section.  

\begin{theorem}\label{thm:two zeroes}
Suppose that $S \subseteq [m] \times [n]$ contains at most two entries in the $m$th row.  Then $\mat_q(m \times n, S, r)$ is equal to a linear combination of similar expressions for smaller matrices with coefficients in $\ZZ[q]$.
\end{theorem}

\begin{corollary}
If $m$, $n$ and $S$ are chosen minimal (in the sense of row- or column-removal) so that $\mat_q(m \times n, S, r)$ is not polynomial in $q$ then $S$ contains at least three entries in each row and column.
\end{corollary}

The rest of this subsection is devoted to the proof of Theorem~\ref{thm:two zeroes}.

\begin{proof}
Consider the $m$th row of the matrix-to-be: it contains zero, one or two forced $0$s (i.e., elements of $S$), and either lies in a dimension-$r$ space spanned by the first $m - 1$ rows or lies outside their $(r-1)$-dimensional span.  We separately compute the number of matrices contributed in each of these three cases.
\begin{compactenum}[1.]
\item[Case 1:] The set $S$ contains no elements in the $m$th row.  We have two possibilities: first, it might be that the first $m - 1$ rows of the matrix span a space of dimension $r$ and the last row lies in this space.  Then we would have $\mat_q((m - 1) \times n, S, r)$ choices for the first $m - 1$ rows and $q^r$ choices for the last row, for a total contribution of $\mat_q((m - 1) \times n, S, r) \cdot q^r$.  Second, it might be that the first $m - 1$ rows span a space of dimension $r - 1$ and the last row lies outside this space.  Then we would have $\mat_q((m - 1) \times n, S, r - 1)$ choices for the first $m - 1$ rows and $q^n - q^{r - 1}$ choices for the last row, for a total contribution of $\mat_q((m - 1) \times n, S, r - 1) \cdot (q^n - q^{r-1})$.  Thus, the total contribution in this case is
\[
\mat_q((m - 1) \times n, S, r) \cdot q^r + \mat_q((m - 1) \times n, S, r - 1) \cdot (q^n - q^{r-1}).
\]
Observe that the two instances of the function $\mat_q$ that appear in this expression involve matrices strictly smaller than $m \times n$.

\item[Case 2:] The set $S$ contains one element in the $m$th row, without loss of generality the entry $(m, n)$.  We have two cases depending on the dimension of the space spanned by the $n$th column, i.e., whether this column is zero or not.
\begin{enumerate}
\item We count matrices in which the $n$th column is the zero vector.  In this case, we may choose the first $m - 1$ rows to span a space of dimension $r$ in $\mat_q((m - 1)\times(n - 1), S, r)$ ways and choose the last row in $q^r$ ways, or we may choose the first $m - 1$ rows to span a space of dimension $r - 1$ in $\mat_q((m - 1)\times(n - 1), S, r - 1)$ ways and choose the last row in $q^{n - 1} - q^{r - 1}$ ways.
\item We count matrices in which the $n$th column is not the zero vector.  In this case, we may choose the first $m - 1$ rows to span a space of dimension $r$ in $\mat_q((m- 1)\times n, S, r) - \mat_q((m - 1)\times(n - 1), S, r)$ ways and (by Proposition~\ref{prop:rank-nullity}) choose the last row in $q^{r - 1}$ ways, or we may choose the first $m - 1$ rows to span a space of dimension $r - 1$ in $\mat_q((m - 1) \times n, S, r - 1) - \mat_q((m - 1)\times(n - 1), S, r - 1)$ ways and choose the last row in $q^{n - 1} - q^{r - 2}$ ways.
\end{enumerate}
The total contribution from these subcases is
\begin{multline*}
\mat_q((m - 1)\times(n - 1), S, r) \cdot q^r + 
\mat_q((m - 1)\times(n - 1), S, r - 1) \cdot (q^{n - 1} - q^{r - 1}) + \\
+ \big(\mat_q((m-1)\times n, S, r) - \mat_q((m - 1)\times(n - 1), S, r)\big)\cdot q^{r - 1} + \\
+ \big(\mat_q((m-1) \times n, S, r - 1) - \mat_q((m - 1)\times(n - 1), S, r - 1) \big)\cdot (q^{n - 1} - q^{r - 2}).
\end{multline*}

\item[Case 3:]  The set $S$ contains two elements in the $m$th row, without loss of generality the elements $(m, n - 1)$ and $(m, n)$.  We have three cases depending on the dimension of the space spanned by the $(n - 1)$th and $n$th columns.
\begin{enumerate}
\item We count matrices in which the last two columns span a space of dimension $0$, i.e., both columns are all zero.  In this case, we may choose the first $m - 1$ rows of the matrix to span a space of dimension $r$ in $\mat_q((m - 1) \times (n - 2), S, r)$ ways and choose the last row to lie in this space in $q^r$ ways, or we may choose the first $m - 1$ rows to span a space of dimension $r - 1$ in $\mat_q((m - 1) \times (n - 2), S, r - 1)$ ways and choose the last row in $q^{n - 2} - q^{r - 1}$ ways.
 
\item We count matrices in which the last two columns span a space of dimension $1$.  There are three possible ways this can come about: the last column may be all zero and the next-to-last column nonzero, the next-to-last column may be all zero and the last column nonzero, or both columns may be nonzero but parallel.  Let 
$S_1$ be the set that results if we remove the next-to-last column from $S$ and replace each element $(i, n)$ in $S$ with $(i, n - 1)$, and let $S_2$ be the set that results if we remove the last two columns from $S$ and add a new entry $(i, n - 1)$ whenever either $(i, n - 1)$ or $(i, n)$ appeared in $S$.  With this notation, the matrices we desire to count fall into the following six classes:
\begin{enumerate}
\item We may choose the first $m - 1$ rows so that they span a space of dimension $r$, the last column is zero and the next-to-last column is nonzero in $\mat_q((m - 1) \times (n - 1), S
, r) - \mat_q((m - 1) \times (n - 2), S
, r)$ ways, and by Proposition~\ref{prop:rank-nullity} we may extend each of these to an $m \times n$ matrix of rank $r$ in $q^{r - 1}$ ways.
\item We may choose the first $m - 1$ rows so that they span a space of dimension $r - 1$, the last column is zero and the next-to-last column is nonzero in $\mat_q((m - 1) \times (n - 1), S
, r - 1) - \mat_q((m - 1) \times (n - 2), S
, r - 1)$ ways, and by Proposition~\ref{prop:rank-nullity} we may extend each of these to a matrix of rank $r$ in $q^{n - 2} - q^{r - 2}$ ways.
\item We may choose the first $m - 1$ rows so that they span a space of dimension $r$, the next-to-last column is zero and the last column is nonzero in $\mat_q((m - 1) \times (n - 1), S_1, r) - \mat_q((m - 1) \times (n - 2), S_1, r)$ ways, and by Proposition~\ref{prop:rank-nullity} we may extend each of these to a matrix of rank $r$ in $q^{r - 1}$ ways.
\item We may choose the first $m - 1$ rows so that they span a space of dimension $r - 1$, the next-to-last column is zero and the last column is nonzero in $\mat_q((m - 1) \times (n - 1), S_1, r - 1) - \mat_q((m - 1) \times (n - 2), S_1, r - 1)$ ways, and by Proposition~\ref{prop:rank-nullity} we may extend each of these to a matrix of rank $r$ in $q^{n - 2} - q^{r - 2}$ ways.
\item We may choose the first $m - 1$ rows so that they span a space of dimension $r$ and the last two columns are nonzero and parallel in $(q - 1)(\mat_q((m - 1) \times (n - 1), S_2, r) - \mat_q((m - 1) \times (n - 2), S_2, r))$ ways, and by Proposition~\ref{prop:rank-nullity} we may extend each of these to a matrix of rank $r$ in $q^{r - 1}$ ways.
\item We may choose the first $m - 1$ rows so that they span a space of dimension $r$ and the last two columns are nonzero and parallel in $(q - 1)(\mat_q((m - 1) \times (n - 1), S_2, r - 1) - \mat_q((m - 1) \times (n - 2), S_2, r - 1))$ ways, and by Proposition~\ref{prop:rank-nullity} we may extend each of these to a matrix of rank $r$ in $q^{n - 2} - q^{r - 2}$ ways.
\end{enumerate}

\item We count matrices in which the last two columns span a space of dimension $2$.  To compute the number of these in which the first $m - 1$ rows span a space of dimension $r$, subtract from $\mat_q((m - 1) \times n, S, r)$ the number of matrices in which the last two columns span a space of dimension less than $2$; this number is computed in cases (a), (b)i, (b)iii and (b)v above.  By Proposition~\ref{prop:rank-nullity}, we may extend each of these to a matrix of rank $r$ in $q^{r - 2}$ ways.  Alternatively, the first $m - 1$ rows may span a space of dimension $r - 1$, and the number of ways in which this happens is the result of subtracting the appropriate values computed in cases (a), (b)ii, (b)iv and (b)vi above.  By Proposition~\ref{prop:rank-nullity}, we may extend each of these to a matrix of rank $r$ in $q^{n - 2} - q^{r - 3}$ ways.
\end{enumerate}

As before, every instance of the function $\mat_q$ in each subcase is applied on matrices of size strictly smaller than $m \times n$.  (We omit the large, uninformative expression that is the total contribution from the cases above.)
\end{compactenum}
Finally, it's easy to check that these cases are exhaustive and that each yields an application of $\mat_q$ on matrices of size smaller than $m \times n$, and that they are combined with coefficients that are polynomials in $q$, as desired.
\end{proof}

\subsubsection{Reductions of this sort can't work if $S$ has three entries per row}\label{sec:obstruction}

On first glance, it appears that the method of proof of Theorem~\ref{thm:two zeroes} can be extended to show that the function $\mat_q(m \times n, S, r)$ is a polynomial in $q$ for any choice of $m$, $n$, $S$ and $r$.  However, as the example of Stembridge \cite{Stem} (see Figure~\ref{fig:Fano}) shows, this is not the case.  In this section, we briefly explain why this recursive approach breaks down for matrices with three or more zeroes per row.

Suppose we attempt to recurse as in the proof of Theorem~\ref{thm:two zeroes} while removing a row with three forced zero entries, e.g., for the set $S$ such that 
\[
\begin{bmatrix}
0     & a_{12} & a_{13} & a_{14} & 0     & 0     \\
0     & a_{22} & a_{23} & 0     & a_{25} & 0     \\
a_{31} & 0     & a_{33} & 0     & 0     & a_{36} \\
a_{41} & a_{42} & 0     & 0     & 0     & a_{46} \\
a_{51} & 0     & 0     & a_{54} & a_{55} & 0     \\
0     & 0     & 0     & a_{64} & a_{65} & a_{66}
\end{bmatrix}
\]
is a representative matrix counted in $\mat_q(6,S,r)$.
For instance, let's remove the bottom row.  Then we wish to count the number of matrices of the form
\[
M' = \begin{bmatrix}
0     & a_{12} & a_{13} & a_{14} & 0     & 0     \\
0     & a_{22} & a_{23} & 0     & a_{25} & 0     \\
a_{31} & 0     & a_{33} & 0     & 0     & a_{36} \\
a_{41} & a_{42} & 0     & 0     & 0     & a_{46} \\
a_{51} & 0     & 0     & a_{54} & a_{55} & 0    
\end{bmatrix}
\]
of each rank, refined by the dimension of the space spanned by the first three columns.  As in Case 3(b) of the proof of Theorem~\ref{thm:two zeroes}, we further refine by the geometric relationship between these columns; it is straightforward to check that if these columns span a space of dimension $0$ or $1$, or if two of them are linearly dependent, then the induction goes through.  The only remaining case is that they span a space of dimension $2$ and are pairwise linearly independent.\footnote{The case that they are independent will simply be the complement of all other cases, so it is nonpolynomial in a minimal example if and only if our selected case is nonpolynomial.}  In this case, we may take an appropriate linear combination of the second and third columns to eliminate the first column, in which case our problem appears at first glance to reduce to the problem of counting matrices of the form
\[
M'' = 
\begin{bmatrix}
a_{12} & a_{13} & a_{14} & 0     & 0     \\
a_{22} & a_{23} & 0     & a_{25} & 0     \\
0     & a_{33} & 0     & 0     & a_{36} \\
a_{42} & 0     & 0     & 0     & a_{46} \\
0     & 0     & a_{54} & a_{55} & 0
\end{bmatrix}
\]
in which the first two columns are linearly independent; we have already shown (in Case 3(c) of the proof of Theorem~\ref{thm:two zeroes}) that under an appropriate inductive hypothesis this yields a polynomial answer in $q$.  The ``catch'' is that the apparently valid reduction is actually wrong: in our particular example, the zeroes in positions $(1, 1)$ and $(2, 1)$ of $M'$, together with the elimination step, impose the condition that the upper-left $2 \times 2$ minor of $M''$ is $0$.  This condition cannot be reduced to the requirement that certain entries be equal to $0$, so any inductive approach of this nature would have to take a substantially stronger inductive hypothesis that allows us to consider restrictions on larger minors.  While this more general question seems potentially interesting, it is quite broad and it seems likely that polynomiality will only hold under very restrictive conditions on the set of selected minors.

\subsubsection{Reduction when $\overline{S}$ has at most two entries in a row}

In this section, we show the complementary result of Theorem~\ref{thm:two zeroes} and give a reduction for $\mat_q(m \times n, S, r)$ when $\overline{S}$ contains at most two entries in some row.

\begin{theorem}\label{thm:two free entries}
Suppose that $S \subseteq [m] \times [n]$ and that $\overline{S}$ contains at most two entries in the $m$th row.  Then $\mat_q(m \times n, S, r)$ is equal to a linear combination of similar expressions for smaller matrices with coefficients in $\ZZ[q]$.
\end{theorem}
\begin{proof}
As before, we split into cases depending on the number of entries of $\overline{S}$ in the $m$th row.
\begin{compactenum}[1.]
\item[Case 1.] The set $\overline{S}$ contains no entries in the $m$th row.  In this case the $m$th row is all zero and we can simply remove it, so $\mat_q(m \times n, S, r) = \mat_q((m - 1) \times n, S, r)$.

\item[Case 2.] The set $\overline{S}$ contains one entry in the $m$th row, without loss of generality the entry $(m, n)$.  In this case, the matrices of rank $r$ with zeroes at the positions marked by $S$ fall into two categories: those for which the $(m, n)$ entry is also $0$, of which there are $\mat_q((m - 1) \times n, S, r)$, and those for which the $(m, n)$ entry is nonzero.  In the latter case we may row-reduce, using the $m$th row to eliminate the $n$th column, and we find that there are $(q - 1)q^a \mat_q((m - 1) \times (n- 1), S, r - 1)$ such matrices, where $a$ is the number of entries of $\overline{S}$ among $\{(1, n), \ldots, (m - 1, n)\}$.  Thus in total we have 
\[
\mat_q((m - 1) \times n, S, r) + (q - 1)q^a \mat_q((m - 1) \times (n- 1), S, r - 1)
\]
matrices in this case.

\item[Case 3.] Suppose that the last row of $\overline{S}$ contains two entries in the $m$th row, without loss of generality the entries $(m, n-1)$ and $(m, n)$.  We count matrices $M$ of rank $r$ whose entries at positions marked by $S$ are equal to $0$, refining by whether the entries $(m, n - 1)$ and $(m, n)$ are also equal to $0$.
\begin{enumerate}
\item If both entries are zero, the number of matrices is just $\mat_q((m - 1) \times n, S, r)$.
\item If the $(m, n)$ entry is nonzero and the $(m, n - 1)$ entry is zero the we are essentially in Case 2 above: the entry $(m, n)$ may be chosen in $(q - 1)$ ways, and we may use the $m$th row to eliminate the $n$th column, which gives a factor of $q^b$, where $b$ is the number of entries of $\overline{S}$ among $\{(1, n), \ldots, (m - 1, n)\}$.  The rest of the matrix may be filled in in $\mat_q((m - 1) \times (n - 1), S, r - 1)$ ways.  So in total we have a contribution of $(q - 1)q^b \mat_q((m - 1) \times (n - 1), S, r - 1)$ in this case.
\item Similarly, if the $(m, n - 1)$ entry is nonzero and the $(m, n)$ entry is zero, we may apply the same technique to eliminate the $(n - 1)$th column, etc.  If we set $c$ to be the number of entries of $\overline{S}$ in $\{(1, n - 1), \ldots, (m - 1, n - 1)\}$ and let $S'$ be the result of removing the $(n - 1)$th column from $S$ and shifting every entry in the $n$th column left, then we have in this case a contribution of $(q - 1)q^c \mat_q((m - 1) \times (n - 1), S', r - 1)$ matrices.
\item Finally, if both the $(m, n - 1)$ and $(m, n)$ entries are nonzero (which may happen in $(q - 1)^2$ ways), we use $(m, n)$ entry to kill nonzero entries (i.e., entries of $\overline{S}$) in the $n$th column.  In this case, nonzero entries get ``transferred'' from the $n$th column to the $(n-1)$th; we pick up an extra factor of $q$ every time there is already an entry of $\overline{S}$ in same row in the $(n - 1)$th column.  Let $S''$ is the set that we get by removing the $(n - 1)$th and $n$th columns from $S$ and adding a new column that has an entry in row $i$ whenever $\{(i, n - 1), (i, n)\} \subseteq S$, and let $d$ be the number of $i \in [m - 1]$ such that neither $(i, n - 1)$ nor $(i, n)$ is in $S$.  In this case, we have a total contribution of $(q - 1)^2 q^d \mat_q((m - 1) \times(n - 1), S'', r - 1)$ matrices.
\end{enumerate}
The total number of matrices in this case is the sum of the expressions in the four subcases.
\end{compactenum}
Each of the expressions appearing above yields a linear combination of applications of $\mat_q$ on matrices of size smaller than $m \times n$, and the coefficients are polynomials in $q$, as desired.
\end{proof}

This result cannot be extended to the case of three nonzero entries in some row, again by the example of \cite{Stem} (see Figure~\ref{fig:Fano}); attempts to follow the same method of proof as in Theorem~\ref{thm:two free entries} meet an obstruction similar to that of Section~\ref{sec:obstruction}.

\begin{corollary}
If $m$, $n$ and $S$ are chosen minimal (in the sense of row- or column-removal) so that $\mat_q(m \times n, S, r)$ is not polynomial in $q$ then $\overline{S}$ contains at least three entries in each row and column.
\end{corollary}

Theorems~\ref{thm:two zeroes} and~\ref{thm:two free entries} allow for the efficient recursive computation of $\mat_q(m \times n, S, r)$ for a large class of sets $S$ when $m$ and $n$ are small (though characterizing precisely which sets $S$ seems hard). The code implemented in \cite{Code} includes these recursions. Unfortunately, for moderately large $m$ and $n$, most sets $S \subseteq [m] \times [n]$ do not admit reductions by our theorems. 

\section{Formula for \texorpdfstring{$\mat_q(n,\overline{B},r)$}{mat(q;n,complement(B),r)} when \texorpdfstring{$B$}{B} has {\sf NE} Property}\label{sec:NE}

In \cite{Hag}, Haglund proved the following result.

\begin{theorem}[{\cite[Thm. 1]{Hag}}]  \label{jh}
For every straight shape $S_{\lambda}$ we have  
\[
\mat_q(n,\overline{S_{\lambda}},r) = (q-1)^r q^{|\lambda|-r}R_r^{\sf (SE)}(S_{\lambda},q^{-1}).
\]
\end{theorem}


We now extend this result (using the same proof technique) to all shapes with the {\sf NE} Property, that is, with the property that for any $i' < i$, $j < j'$, if $(i, j)$, $(i', j)$ and $(i, j')$ belong to $B$, then $(i', j')$ does as well.

\begin{theorem} \label{NEthm}
Fix any $n$ and $r$ and any set $B \subseteq [n] \times [n]$ with the {\sf NE} Property.  The number of $n\times n$ matrices over $\F$ of rank $r$ whose support is contained in $B$ is 
\begin{equation} \label{NEcase}
\mat_q(n,\overline{B},r) = (q-1)^r q^{\#B-r}R_r^{\sf (NE)}(B,q^{-1}). 
\end{equation}
\end{theorem}

\begin{proof}
Choose a matrix 
  $A$ counted in $\mat_q(n, \overline{B},r)$, that is, whose support 
  is in $B$, and perform Gaussian elimination in the following (north-east)
  order: traverse each column from bottom to top, starting with the leftmost (i.e., first) column. 
  When you come to a nonzero entry (i.e., a pivot), use it to eliminate the 
  entries to its north in the same column and to its east in the same row.  See Figure~\ref{NEelim} for an example of this stage of the elimination process.
  Then move on to the next column and repeat until there is at most one nonzero entry in every row and column. 

By the {\sf NE} Property, at each stage of the elimination process just described we obtain another matrix counted in $\mat_q(n,
  \overline{B},r)$. After elimination, the positions of the pivots are a placement of $r$ non-attacking rooks on $B$.

\begin{figure}
\begin{center}
\includegraphics[height=3.8cm]{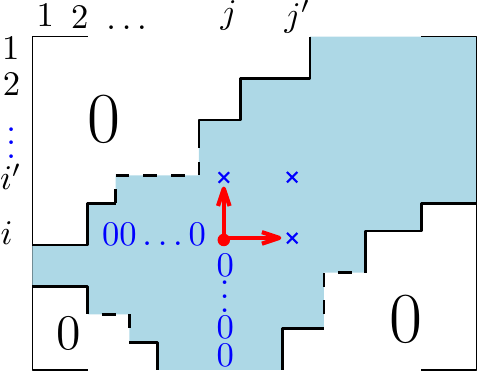}
\caption{NE elimination on a representative matrix counted in $\mat_q(n, \overline{B},r)$ with a pivot on $(i,j)$ where $B$ has the {\sf NE} Property.}
\label{NEelim}
\end{center}
\end{figure}

Given a fixed placement of $r$ non-attacking rooks on $B$, let $a$ be the number of cells in $B$ that are directly north or directly east of a rook. There are $(q-1)^rq^a$ matrices of rank $r$ whose support is in $B$ that give this placement after the elimination procedure described above. It is not hard to see that $a=\#B-r-\inv_{\sf NE}(C,B)$.  
Thus, summing over all placements or $r$ non-attacking rooks, we obtain
\[
\mat_q(n,\overline{B},r)=(q-1)^rq^{\#B-r}\sum_C(q^{-1})^{\inv_{\sf NE}(C,B)} =(q-1)^rq^{\#B-r}R_r^{\sf(NE)}(B,q^{-1}),
\]
as desired.
\end{proof}

Note that {\em a priori} it is not clear that the expression on the right-hand side of Equation~\eqref{NEcase} is a polynomial. However, this expression is a polynomial for the following reason: for any rook placement, there cannot be any more inversions than there are empty cells without rooks in them.  There are $\#B$ cells unoccupied by zeros, and, of these, $r$ have rooks in them.  So the maximum value of $\inv_{\sf NE}(C,B)$ is $\#B-r$.  Since this is the power of $q$ at the beginning of the formula, there will not be any $q^{-1}$ terms, and $\mat_q(n,\overline{B},r)$ is a polynomial.

\begin{example}
For $n=4$ and $r=4$, the set $B=([4]\times [4]) \setminus \{(1,1),(3,4),(4,1),(4,3),(4,4)\}$ has the {\sf NE} Property (as in Figure~\ref{exNE}(a)) and there are three placements of four rooks on $B$ (as in Figure~\ref{exNE}(b)). The number of NE-inversions of these placements are $0,\,1$ and $1$ respectively. Thus 
\begin{align*}
\mat_q(4,\overline{B},4) &=(q-1)^4q^{11-4}(1+2q^{-1})\\ 
&= (q-1)^4(q^7+2q^6).
\end{align*}

\begin{figure}
\begin{center}
\includegraphics[height=3.3cm]{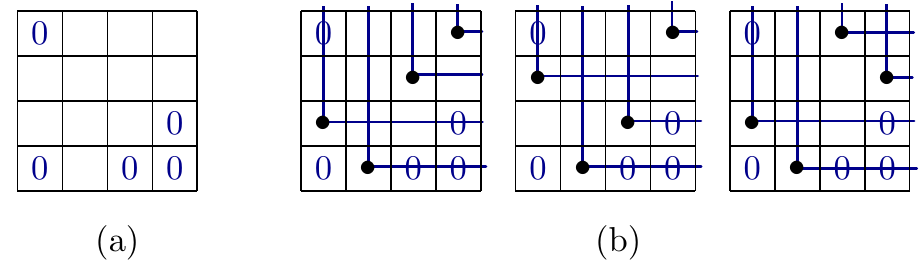}
\caption{(a) Set $B$ with the {\sf NE} Property. (b) Example of computing $\mat_q(n,\overline{B},r)$ when $B$ has the {\sf NE} Property. There are three placements of four rooks in $B$ with 0, 1 and 1  NE-inversions respectively. By Theorem~\ref{NEthm}, $\mat_q(4,\overline{B},4)=(q-1)^4q^{11-4}(1+2q^{-1})$.}
\label{exNE}
\end{center}
\end{figure}
\end{example}

We give two corollaries of Theorem~\ref{NEthm}.  First, since a straight shape $S_{\lambda}$ has the {\sf NE} Property, by comparing Haglund's result and Theorem~\ref{NEthm} we see that the {\sf (NE)} and {\sf (SE)} $q$-rook numbers of $S_{\lambda}$ agree.
\begin{corollary} \label{SENErook}
For any straight shape $S_{\lambda}$ we have
\[
R_r^{\sf (NE)}(S_{\lambda},q)= R_r^{\sf (SE)}(S_{\lambda},q).
\]
\end{corollary}
(Recall that in general the {\sf (NE)} and {\sf (SE)} $q$-rook numbers of a general board do not agree; see for example Remark~\ref{SENE}.)

Second, since any skew shape $S_{\lambda/\mu}$ has the {\sf NE} Property, we have the following corollary:
\begin{corollary} \label{skewcor}
For any skew shape $S_{\lambda/\mu}$,
\[\mat_q(n,\overline{S_{\lambda/\mu}},r)=(q-1)^r f(q),\]
where $f(q)$ is a polynomial with nonnegative integer coefficients.
\end{corollary}

\begin{example}
For $\lambda/\mu=4432/31$, we have
\begin{align*}
\mat_q(4\times 4,\overline{S_{4432/31}},3) & = (q-1)^3q^{9-3}(2q^{-1}+8q^{-1} + 7q^{-3} + q^{-4})\\
& = (q-1)^3 q^2(q+1)(2q^2+6q+1).
\end{align*}
In general, for skew shapes $S_{\lambda/\mu}$ there is no product formula for $\mat_q(n,\overline{S_{\lambda/\mu}},r)$ analogous to~\eqref{GRhooknum}, even when $r=n$.
\end{example}

\section{Studying \texorpdfstring{$\mat_q(n,S,r)$}{mat(q;n,S,r)} when \texorpdfstring{$S$}{S} is a Rothe diagram}\label{sec:rothe}

Given a permutation $w\in \mathfrak{S}_n$  written as a word $w=w_1w_2\cdots w_n$ where $w_i$ is the image of $w$ at $i$, the {\bf Rothe diagram} $R_w$ is the set
\[
R_w = \{(i,j) \mid 1\leq i, j\leq n,\,\, w(i)>j,\,\, w^{-1}(j)>i\}.
\]  
Equivalently $R_w$ is the set of elements in $[n]\times [n]$ that do not lie directly south or directly east of entries $(i,w_i)$ of the permutation matrix of $w$. See Figure~\ref{fig:exRothe} for some examples of Rothe diagrams. Note that $\#R_w$ is the number of inversions of $w$, that is, the number of pairs $(i,j)$ such that $i<j$ but $w_i>w_j$. Also, $R_w$ has the following property: if $(i,j)$ and $(k,\ell)$ are in $R_w$ and $i>k$, $j<\ell$ then the entry $(k, j)$ is also in $R_w$. We call this the {\bf Le property} of Rothe diagrams.\footnote{The name ``Le'' was invented in \cite[Sec.\ 6]{AP} in a context where the three entries in question formed a backwards letter ``L''; here we keep this terminology even though for $R_w$ the three entries form instead the letter ``$\Gamma$''.}

\begin{figure}
\[
\begin{array}{ccc}
w=41523 & w=21534 & w=31524\\
\left[ \begin {array}{ccccc} 0&0&0&\textcolor{red}{\underline{a_{{14}}}}&a_{{15}}
\\ \textcolor{red}{\underline{a_{{21}}}}&a_{{22}}&a_{{23}}&a_{{24}}&a_{{25}}
\\ a_{{31}}&0&0&a_{{34}}&\textcolor{red}{\underline{a_{{35}}}}
\\ a_{{41}}&\textcolor{red}{\underline{a_{{42}}}}&a_{{43}}&a_{{44}}&a_{{45}}
\\ a_{{51}}&a_{{52}}&\textcolor{red}{\underline{a_{{53}}}}&a_{{54}}&a_{{55}}
\end {array} \right]
&
\left[ \begin {array}{ccccc} 0&\textcolor{red}{\underline{a_{{12}}}}&a_{{13}}&a_{{14}}&a_{{15}
}\\ \textcolor{red}{\underline{a_{{21}}}}&a_{{22}}&a_{{23}}&a_{{24}}&a_{{25}
}\\ a_{{31}}&a_{{32}}&0&0&\textcolor{red}{\underline{a_{{35}}}}
\\ a_{{41}}&a_{{42}}&\textcolor{red}{\underline{a_{{43}}}}&a_{{44}}&a_{{45}}
\\ a_{{51}}&a_{{52}}&a_{{53}}&\textcolor{red}{\underline{a_{{54}}}}&a_{{55}}
\end {array} \right] 
&
\left[ \begin {array}{ccccc} 0&0&\textcolor{red}{\underline{a_{{13}}}}&a_{{14}}&a_{{15}}
\\ \textcolor{red}{\underline{a_{{21}}}}&a_{{22}}&a_{{23}}&a_{{24}}&a_{{25}}
\\ a_{{31}}&0&a_{{33}}&0&\textcolor{red}{\underline{a_{{35}}}}
\\ a_{{41}}&\textcolor{red}{\underline{a_{{42}}}}&a_{{43}}&a_{{44}}&a_{{45}}
\\ a_{{51}}&a_{{52}}&a_{{53}}&\textcolor{red}{\underline{a_{{54}}}}&a_{{55}}
\end {array} \right] 
\end{array}
\]
\caption{Representative matrices counted by $\mat_q(5,R_w,r)$ where $R_w$ is a Rothe diagram and $w$ is (i) $41523$ (vexillary), (ii) $21534$ (skew-vexillary), (iii) $w=31524$ (not skew-vexillary). The entries $a_{i\,w_i}$ are in \textcolor{red}{\underline{red}}.}
\label{fig:exRothe}
\end{figure}

The main conjecture for Rothe diagrams, which has been verified for $n\leq 7$ using the results in Section~\ref{sec:general conditions} \cite{Code}, is the following:

\begin{conjecture} \label{conj:rothe}
If $R_w$ is the Rothe diagram of a permutation $w$ in $\mathfrak{S}_n$ and $0\leq r\leq n$ then $\mat_q(n,R_w,r)/(q-1)^r$ is a polynomial in $q$ with nonnegative integer coefficients.
\end{conjecture}

In Subsection~\ref{properties} we give properties of Rothe diagrams that help in calculating $\mat_q(n,R_w,r)$. 
In Subsection~\ref{skew} we study Conjecture~\ref{conj:rothe} for the family of permutations $w$ such that $R_w$ is the complement of a skew shape (after permuting rows and columns).  The conjecture holds for such permutations by Corollary~\ref{skewcor}. In Theorem~\ref{thm:indecomp}, we characterize these permutations.  


\subsection{Properties of \texorpdfstring{$\mat_q(n,S,r)$}{mat(q;n,S,r)} when \texorpdfstring{$S$}{S} is a Rothe diagram} \label{properties}

In this section we give some simple properties of $\mat_q(n,S,r)$ when $S=R_w$ is the Rothe diagram of a permutation $w$. These properties are useful to simplify the size of computations involved in empirically confirming conjectures about $\mat_q(n,R_w,r)$ like Conjectures~\ref{conj:rothe} and \ref{conj:PoinRothe}.

If the permutation $w$ is the word $w_1w_2\cdots w_n$, the {\bf reverse} of $w$ is the permutation $re(w)=w_nw_{n-1}\cdots w_1$. The {\bf complement} of $w$ is the permutation $c(w)=u_1u_2\cdots u_n$ where $u_i=n+1-i-w_i$. In addition, the {\bf reverse complement} of $w$ is the permutation $rc(w)=v_1v_2\cdots v_n$ where $v_i=n+1-w_{n+1-i}$. Lastly, the {\bf left-to-right maxima} of $w$ are the values $w_i$ such that $w_i > w_j$ for all $j$ such that $1\leq j < i$. 

\begin{proposition}
Given a permutation $w$ in $\mathfrak{S}_n$ and its Rothe diagram $R_w$, we have
\begin{compactenum} [\rm (i)]
\item $\mat_q(n,R_w,r)=\mat_q(n,R_{w^{-1}},r)$ and
\item $\mat_q(n, R_w,r) = \mat_q(n, R_{rc(w)},r)$.
\end{compactenum}
\end{proposition}

\begin{proof}
It is easy to see that for any permutation $w$, the diagram $R_{w^{-1}}$ is the transpose of $R_w$, and the first statement follows immediately. We now consider the second statement.

Fix a permutation $w$ with Rothe diagram $R_w$. Each element $(i,j)$ of $R_w$ corresponds to the inversion of $w$ formed by the entries with matrix coordinates $(i,w_i)$ and $(w^{-1}_j,j)$.  In $rc(w)$, these elements of $w$ are transformed to $(n + 1 - i, n + 1 - w_i)$ and $(n + 1 - w^{-1}_j, n + 1 - j)$ and still form an inversion; in $R_{rc(w)}$, this inversion corresponds to the element with coordinates $(n+1-w^{-1}_j,n+1-w_i)$. It follows immediately that the diagram $R_{rc(w)}$ is the result of taking the transpose of $R_w$, rearranging rows and columns by multiplying on both sides by the permutation matrix of $w$, and rotating the result by $180^{\circ}$.
\end{proof}

Next we characterize the indices of the rows and columns of $[n]\times [n]$ entirely contained in $\overline{R_w}$. This is useful for computation because it is easy to express $\mat_q(n,S,r)$ in terms of values of $\mat_q$ for sets obtained by removing  rows or columns that contain no elements of $S$.

\begin{proposition}\label{LRMcol}
The $k$th column (row) of $[n]\times [n]$ is contained in $\overline{R_w}$ if and only if $k$ is a left-to-right maximum of $w$ (of $w^{-1}$).
\end{proposition}

\begin{proof}
This follows from the definitions of $R_w$ and of the left-to-right maxima.
\end{proof}

\subsection{Skew-vexillary permutations} \label{skew}

A permutation $w$ is {\bf vexillary} if its Rothe diagram, up to a permutation of its rows and columns, is the diagram of a partition. Call this partition $\lambda(w)$. Then by Haglund's Theorem \ref{jh}, for vexillary permutations $w$ we have that
\[
\mat_q(n, R_w,r) = \mat_q(n,S_{\lambda(w)},r)=(q-1)^rq^{n^2-\inv(w)-r}R_r^{\sf (NE)}(\overline{S_{\lambda(w)}},q^{-1}).
\]
 It is well-known that $w$ is vexillary if and only if $w$ avoids $2143$ \cite{LS}, i.e., there is no sequence $1\leq i<j<k<l\leq n$ in $w$ such that $w_j<w_i<w_l<w_k$.

Next we give a characterization of permutations whose Rothe diagram, up to a permutation of rows and columns, is the complement of a skew shape. For such a permutation $w$, we have by Corollary~\ref{skewcor} that $\mat_q(n,R_w,r)/(q-1)^r$ is a polynomial with nonnegative integer coefficients. So Conjecture~\ref{conj:rothe} holds for these permutations.

For the proof we need the following definition: we say that a skew shape $S_{\lambda/\mu}$ in $[n]\times [n]$ is {\bf non-overlapping}  if there is no row nor column that contains entries from both $S_\mu$ and $\overline{S_\lambda}$.

\begin{theorem} \label{thm:indecomp}
The Rothe diagram of $w=w_1w_2\cdots w_n$ is, up to permuting its rows and columns, the complement of a skew shape if and only if $w$ can be decomposed as $a_1a_2\cdots a_k b_1b_2\cdots b_{n-k}$ where $a_i<b_j$ and each of $a_1a_2\cdots a_k $ and $b_1b_2\cdots b_{n-k}$  is $2143$-avoiding.  
\end{theorem}

\begin{proof}
First we prove the ``if'' direction. This argument is illustrated in Figure~\ref{fig:skewvex1}.  Suppose that $w$ can be decomposed into $a=a_1a_2\cdots a_k$ and $b=b_1b_2\cdots b_{n-k}$ as in the theorem statement.  Then the Rothe diagram $R_w$ is block-diagonal, i.e., it consists of some entries in the upper-left $k \times k$ block and some in the lower-right $(n-k)\times (n-k)$ block, with no entries in the upper-right $k\times (n-k)$ block or lower-left $(n-k) \times k$ block.  Furthermore, note that the  upper-left and lower-right subdiagrams are identical to the Rothe diagrams of the permutations order-isomorphic to $a_1a_2\cdots a_k$ and $b_1b_2\cdots b_{n-k}$, respectively.  

Since both of these permutations are $2143$-avoiding, and their Rothe diagrams in the upper-left and lower-right corners do not share any rows or columns in common, they can be rearranged independently to form two separate straight shapes.  We may then rotate the straight shape corresponding to $R_b$ by $180^\circ$ via permuting rows and columns (without changing the rearranged upper-left corner) to get a straight shape in the upper-left corner and an upside-down straight shape in the lower-right. This is the outside of a skew shape, as desired.

\begin{figure}
\begin{center}
\includegraphics[height=3cm]{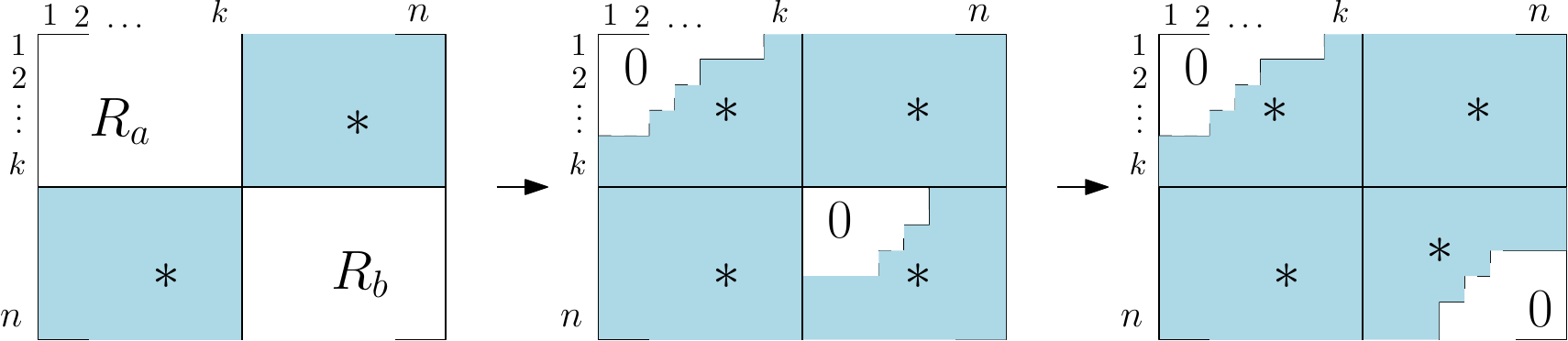}
\end{center}
\caption{If $w$ can be decomposed as $a_1a_2\cdots a_k b_1b_2\cdots b_{n-k}$ where $a_i<b_j$ and both $a=a_1a_2\cdots a_k$ and $b=b_1b_2\cdots b_{n-k}$ are $2143$ avoiding then $\overline{R_w}$ can be rearranged into a skew shape.}
\label{fig:skewvex1}
\end{figure}

Second, we prove the ``only if'' direction of the theorem. Suppose that the diagram $R_w$, when rearranged, forms the complement of a skew shape $S_{\lambda/\mu}$. This skew shape contains the column that was previously (i.e., before rearrangement) given by $\{(j,w_1) \mid j \geq 1\}$. Likewise, it contains the row that was previously given by  $\{(w^{-1}_1,j) \mid j \geq 1\}$. It follows that the skew shape $S_{\lambda/\mu}$ is non-overlapping. After rearrangement, every entry of $R_w$ either belongs to $S_{\mu}$ or $\overline{S_{\lambda}}$. We use this partition of the elements of $R_w$ to identify the appropriate decomposition of $w$.

We color an entry of $R_w$ blue if it belongs to $S_{\mu}$ after rearrangement, otherwise we color it red. We show the following claim: for every entry $w_i$ of $w$, the elements of $R_w$ in the same row or column as $(i,w_i)$ are either all blue or all red.

Since $S_{\lambda/\mu}$ is non-overlapping, the entries of $R_w$ in each row have the same color, and likewise for columns. If there is an entry $(i,w_i)$ with elements  $(i,j)$ and $(k,w_i)$ of $R_w$ then by the Le property of Rothe diagrams $(k,j)$ is also in $R_w$. Therefore all three entries have the same color, and the claim follows.

By the argument of the preceding paragraph,  we may color the elements of $w$ as follows: if $(i,w_i)$ is in the same row or column as a red entry of $R_w$ then we color $w_i$ red, whereas if $(i,w_i)$ is in the same row or column as a blue entry of $R_w$ then we color $w_i$ blue, and otherwise we leave $w_i$ uncolored. We observe a few properties of the colored and uncolored elements of the permutation. First, inversions of $w$ can only happen between elements of the same color. Second, $w_i$ is uncolored if and only if $w_i$ is not involved in any inversions. And third, the subword of the blue (respectively, red) elements of $w$ is $2143$-avoiding. This is because by definition, the entries of $R_w$ in the same row or column as $(i,w_i)$ for blue (respectively, red) $w_i$ are exactly the entries in $S_{\mu}$ (respectively, $\overline{S_{\lambda}}$) after rearrangement. This is equivalent to saying that the subword of the blue (respectively, red) elements of $w$ is vexillary and thus $2143$-avoiding. 

From the three observations above it follows that the permutation $w$ decomposes as $u_1c_1u_2c_2u_3$ where (i) the $u_i$ are (possibly empty) blocks of uncolored elements, $c_1$ is the block of elements of one color of $w$, and $c_2$ is the block of elements of the other color of $w$; (ii) the entries of each block are smaller than the entries of the following blocks, and (iii) the blocks $c_1$ and $c_2$ are $2143$-avoiding. Finally, if we set $a_1a_2\cdots a_k = u_1c_1$ and $b_1b_2\cdots b_{n-k}=u_2c_2u_3$ we get a desired decomposition of $w$ where $a_i<b_j$ and  $a_1a_2\cdots a_k$ and $b_1b_2\cdots b_{n-k}$ are $2143$-avoiding.
\end{proof}

We call the above permutations {\bf skew-vexillary}\footnote{Note that in the literature \cite[Prop. 2.3]{BJS} there is another meaning of the term ``skew-vexillary permutation'' which does not seem to be related to our definition.} and we denote by $\lambda/\mu(w)$ the skew shape whose complement is the rearrangement of $R_w$.\footnote{The ``function'' $\lambda/\mu(w)$ is not actually well-defined most of the time since you can switch the upper-left and lower-right corners by permuting rows and columns. Luckily nothing we use it for depends on this choice.} 

\begin{corollary}\label{cor:skewvex}
By Theorem~\ref{NEthm}, if $w$ is skew-vexillary then $\mat_q(n, R_w, r)/(q-1)^r$ is equal to $q^{n^2-\inv(w)-r}R_r^{\sf (NE)}(S_{\lambda/\mu(w)},q^{-1})$, a polynomial with nonnegative integer coefficients. In particular, Conjecture~\ref{conj:rothe} holds for skew-vexillary permutations. 
\end{corollary}

If $w$ is a skew-vexillary permutation then every subpermutation of $w$ is, as well.  This implies that skew-vexillarity may be rephrased as a pattern-avoidance condition.  We do this now.

\begin{proposition}\label{prop:patterns}
The permutation $w \in \mathfrak{S}_n$ can be decomposed as $w = a_1 \cdots a_k b_{1} \cdots b_{n - k}$ such that $a_i < b_j$ for all $i$ and $j$ and the permutations $a_1 \cdots a_k$ and $b_{1} \cdots b_{n - k}$ avoid $2143$ if and only if $w$ avoids the nine patterns $24153$, $25143$, $31524$, $31542$, $32514$, $32541$, $42153$, $52143$ and $214365$.
\end{proposition}

\begin{proof}
Call the decomposition in question an ``SV-decomposition'' (for {\bf\color{red} S}kew-{\bf\color{red} V}exillary).  First, we show that if $w$ contains any of the nine patterns listed in the statement of the theorem, it does not have an SV-decomposition.

Let $p$ be any of the eight patterns of length $5$; it's easy to check that $p$ is indecomposable, i.e., we cannot write $p = uv$ with $u$, $v$ nonempty and $u_i < v_j$ for all $i$, $j$.  Thus, if we write $w = ab$ with $a_i < b_j$ we must have either $p$ contained in $a$ or $p$ contained in $b$.  Since $p$ contains $2143$, it follows that either $a$ or $b$ contains $2143$, so this decomposition is not SV, as desired.

Now consider the case of the pattern $214365$.  Any decomposition of $w$ decomposes $214365$, and it's easy to see that in any of the four decompositions of $214365$, one piece or the other contains a copy of $2143$.  This completes the proof that any permutation containing the given patterns has no SV-decomposition.

Now consider the converse.  Suppose that $w$ is not SV-decomposable.  There are two cases.

If $w$ is indecomposable and contains $2143$, then $w$ contains a minimal indecomposable permutation that contains $2143$.  The minimal $2143$-containing indecomposable permutations are precisely the eight permutations of length $5$ that we consider.

Finally, we show by induction that every decomposable but not SV-decomposable permutation contains one of the nine patterns.  Choose a  such $w$, and write $w = ab$ with $a_i < b_j$.  Without loss of generality, in this decomposition we have that $a$ contains $2143$.  If $b$ has a descent, it follows immediately that $w$ contains $214365$.  Otherwise, $w = a_1 \cdots a_k (k + 1)(k + 2) \cdots n$.  Observe that a permutation of this form has an SV-decomposition if and only if the shorter permutation $a = a_1 \cdots a_k$ has an SV-decomposition; thus, $a$ has no SV-decomposition.  If $a$ is indecomposable, we have by the preceding paragraph that $a$ contains one of the nine patterns; if $a$ is decomposable, we have the same result by induction.

Putting the two cases together, every permutation that is not SV-decomposable contains at least one of the nine patterns, as desired.
\end{proof}

\begin{remark} \label{rem:stansymmfcn}
Vexillary permutations have many other interesting properties (see for example  \cite[Sections 2.6.5 and 2.8.1]{LM}). For example, the {\em Stanley symmetric function} \cite[Section 2]{StCox} $F_w$ of a vexillary permutation $w$ equals the Schur function $s_{\lambda(w)}$. Since a skew-vexillary permutation $w$ with skew shape $\lambda/\mu(w)$ is SV-decomposable then $F_w$ is the product $s_{[n]\times [n] \backslash \lambda}s_{\mu}$ of Schur functions. Do other properties of vexillary permutations carry over to skew-vexillary permutations? 
\end{remark}

\subsubsection{An enumerative aside}

Any structurally interesting class of permutations calls out to be enumerated.  Vexillary permutations were enumerated by West \cite{W} (who showed they are in bijection with $1234$-avoiding permutations, which had been enumerated earlier by Gessel \cite{IG}). We now enumerate skew-vexillary permutations in terms of the generating function for vexillary permutations. For convenience, we denote by $\mathfrak{S}_n(2143)$ the set of vexillary permutations of length $n$. Also, given  two permutations $u \in \mathfrak{S}_i$ and $v \in \mathfrak{S}_j$, we denote by $w = u \oplus v$ the permutation in $\mathfrak{S}_{i + j}$ defined by $w_t = u_t$ for $t \in [i]$ and $w_{t + i} = i + v_t$ for $t \in [j]$.

\begin{theorem} \label{thm:genseries}
Let $V(x)$ be the ordinary generating function for $2143$-avoiding permutations and let $SV(x)$ be the ordinary generating function for skew-vexillary permutations.  We have
\[
SV(x) = (1 - x)V(x)^2 - V(x) + \frac{1}{1 - x}.
\]
\end{theorem}
\begin{proof}
Let $V(x) = 1 + x + 2x^2 + 6x^3 + 23x^4 + \ldots$ be the \emph{ordinary} generating function for $2143$-avoiding permutations and let $I(x) = x^2 + 3x^3 + 13x^4 + \ldots$ be the generating function for indecomposable $2143$-avoiding permutations of length $2$ or more.  Suppose $w \in \mathfrak{S}_n(2143)$ can be written $w = u_1 \oplus u_2$ where $u_1$ and $u_2$ are nonempty permutations.  If both $u_1$ and $u_2$ contain an inversion, the four entries from the two inversions form an instance of $2143$ in $w$, a contradiction.  Thus, one of $u_1$ and $u_2$ must be the identity permutation.  It follows immediately that every $w \in \mathfrak{S}_n(2143)$ other than the identity can be written as $w = \id_i \oplus u \oplus \id_k$ for identity permutations $\id_i$ and $\id_k$ (possibly of length $0$) and an indecomposable $2143$-avoiding permutation $u$.  Thus,
\[
V(x) = \frac{1}{1 - x} + \frac{I(x)}{(1 - x)^2}
\]
and so 
\[
I(x) = (1 - x)^2 V(x) + x - 1.
\]

Now suppose that $w \in \mathfrak{S}_n$ has the property that $w = u \oplus v$ where $u$ and $v$ are $2143$-avoiding permutations (possibly of length $0$).  It follows from the preceding analysis that either $w$ is $2143$-avoiding itself or we can write $w = \id_i \oplus u' \oplus \id_j \oplus v' \oplus \id_k$ where the $\id$s are identity permutations (possibly of length $0$) and $u'$ and $v'$ are indecomposable $2143$-avoiding permutations of length $2$ or more.  Thus, the generating function $SV(x)$ for such permutations is given by 
\begin{align*}
SV(x) & = V(x) + \frac{I(x)^2}{(1 - x)^3} \\
& = (1 - x) V(x)^2 - V(x) + \frac{1}{1 - x},
\end{align*}
as desired.
\end{proof}

If $w$ is skew-vexillary then $\mat_q(n,R_w,r)/(q-1)^r$ is of the form $q^{n^2-\inv(w)-r}\sum_{\text{some } u \in \mathfrak{S}_n} q^{-\inv(u)}$ by Corollary~\ref{cor:skewvex}. Another polynomial with this form is the {\em Poincar\'e polynomial} of the {\em strong Bruhat order} (see e.g. \cite[Sec. 2.1.2]{LM}) in $\mathfrak{S}_n$. In the next subsection we study the connections between these and $\mat_q(n,R_w,n)$.

\section{Poincar\'e polynomials, \texorpdfstring{$\mat_q(n,R_w,n)$}{mat(q;n,Rw,n)} and \texorpdfstring{$q$}{q}-rook numbers}\label{sec:poincare}
A natural question when faced with a family of polynomials with positive integer coefficients is whether they count some nice combinatorial object.  In this section, we investigate connections between our polynomials $\mat_q(n, R_w, n)$ (note in particular that we focus on the case of full rank) and certain well-known polynomials we define now. 

As before, let $\inv(w)$ denote the number of inversions $\#\{(i,j) \mid i<j, w_i>w_j\}$ of $w$.  Recall the notion of the {\bf strong Bruhat order} $\prec$ on the symmetric group \cite[Ch.\ 2]{BB}: if $t_{ij}$ is the transposition that switches $i$ and $j$, we have as our basic relations that $u \prec u\cdot t_{ij}$ in the strong Bruhat order when $\inv(u) + 1 = \inv(u \cdot t_{ij})$, and we extend by transitivity. Let $P_w(q)=\sum_{u \succeq  w} q^{\inv(u)}$ be the (upper) {\bf Poincar\'e polynomial} of $w$, where we sum over all permutations $u$ that succeed $w$ in the strong Bruhat order. Equivalently, $P_w(q)$ is the rank generating function of the interval $[w,w_0]$ in the strong Bruhat order where $w_0$ is the largest element $n\,n-1\,\cdots\,21$ of this order. 

\begin{example}
If $w=3412$, then the permutations in $\mathfrak{S}_4$ that succeed $w$ in the Bruhat order are $3412,3421,4312$ and $4321$. The generating polynomial for this set by number of inversions is $P_{2143}(q)=q^6+2q^5+q^4$. 
\end{example}

In \cite{Sj}, Sj\"ostrand gave necessary and sufficient conditions for $P_w(q)$ to be equal to a $q$-rook number of a skew shape associated to $w$. Namely, the {\bf left hull} $H_L(w)$ of $w$ is the smallest skew shape that covers $w$.  Equivalently, $H_L(w)$ is the union over non-inversions $(i,j)$ of $w$ of the rectangles $\{ (k,\ell) \mid w_i \leq k \leq w_j,   i \leq \ell \leq j\}$. 
See Figure~\ref{exhull} for an example of the left hull of a permutation.

\begin{figure}
$$
\begin{array}{cc}
R_{35142} &H_L(35142) 
\\ 
\left[ \begin {array}{ccccc} 0&0&\textcolor{red}{\underline{a_{{13}}}}&a_{{14}}&a_{{15}}
\\ 0&0&a_{{23}}&0&\textcolor{red}{\underline{a_{{25}}}}\\ \textcolor{red}{\underline{a_{
{31}}}}&a_{{32}}&a_{{33}}&a_{{34}}&a_{{35}}\\ a_{
{41}}&0&a_{{43}}&\textcolor{red}{\underline{a_{{44}}}}&a_{{45}}\\ a_{{51}}&\textcolor{red}{\underline{a
_{{52}}}}&a_{{53}}&a_{{54}}&a_{{55}}\end {array} \right] 
&
\left[ \begin {array}{ccccc} 0&0&\textcolor{red}{\underline{a_{{13}}}}&a_{{14}}&a_{{15}}
\\ 0&0&a_{{23}}&a_{{24}}&\textcolor{red}{\underline{a_{{25}}}}
\\ \textcolor{red}{\underline{a_{{31}}}}&a_{{32}}&a_{{33}}&a_{{34}}&0
\\ a_{{41}}&a_{{42}}&a_{{43}}&\textcolor{red}{\underline{a_{{44}}}}&0
\\ a_{{51}}&\textcolor{red}{\underline{a_{{52}}}}&0&0&0\end {array} \right] 
\end{array}
$$
\caption{Matrices indicating the (i) Rothe diagram and (ii) left hull of $w = 35142$.  The matrix entries $a_{i\,w_i}$ are in \textcolor{red}{\underline{red}}.}
\label{exhull}
\end{figure}

The following special case of a result by Sj\"ostrand characterizes when $P_w(q)$ is equal to the rook polynomial of the left hull of the permutation $w$. 

\begin{theorem}[{\cite[Cor. 3.3]{Sj}}] \label{sjothm}
The Bruhat interval $[w,w_0]$ in $\mathfrak{S}_n$ equals the set of rook placements in the left hull $H_L(w)$ of $w$ (and in particular $R_n^{\sf (SE)}(H_L(w),q)=q^{|\mu|}P_w(q)$ where $\mu$ is the shape such that $H_L(w)=S_{\lambda/\mu}$ for some $\lambda$) if and only if $w$ avoids the patterns $1324, 24153, 31524$, and $426153$. 
\end{theorem}

If $w$ is a skew-vexillary then by Corollary~\ref{cor:skewvex} we know that $\mat_q(n,R_w,n)/(q-1)^n$ is  (up to a power of $q$) a $q$-rook number. Next we show that this $q$-rook number is essentially a $q$-rook number of the left hull of a permutation $v$ that avoids the four patterns above. Therefore by Theorem~\ref{sjothm} $\mat(n,R_w,n)/(q-1)^n$ is (up to a power of $q$) a Poincar\'e polynomial $P_v(q)$.


\subsection{\texorpdfstring{$\mat_q(n,R_w,n)$}{mat(q;n,Rw,n)} for skew-vexillary permutations is a Poincar\'e polynomial}

In this section we use Sj\"ostrand's result (Theorem~\ref{sjothm}) to  show that for skew-vexillary permutations $w$, the function $\mat_q(n,R_w,n)/(q-1)^n$ is not only a polynomial with nonnegative coefficients but, up to a power of $q$, is a Poincar\'e polynomial.

\begin{proposition} \label{prop:MRP}
If $w$ is skew-vexillary then 
\[
\mat_q(n, R_w, n) = q^{\binom{n}{2}-\inv(w)}(q-1)^n\cdot P_v(q)
\]
for some permutation $v \in \mathfrak{S}_n$.
\end{proposition}

\begin{figure}
\begin{center}
\includegraphics[height=3.6cm]{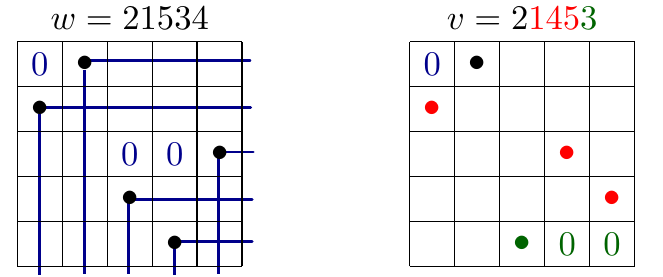}
\end{center}
\caption{Example of Proposition~\ref{prop:MRP}. For the permutations $w$ and $v$ shown we have that $\mat_q(5,R_w,5)/(q-1)^5 = q^{\binom{5}{2}-\inv(w)}\cdot P_v(q)$.}
\label{fig:exSkewPoin}
\end{figure}

\begin{proof}

If $w$ is skew-vexillary, then by Corollary~\ref{cor:skewvex} we know that 
\[
\mat_q(n, R_w,n)/(q-1)^n  =  q^{n^2 - \inv(w) - n}R_n^{\sf (NE)}(S_{\lambda/\mu(w)},q^{-1}),
\] 
where $R^{\sf (NE)}(S_{\lambda/\mu(w)},q)$ is the rook polynomial of $S_{\lambda/\mu(w)}$, the non-overlapping skew shape whose complement is the rearrangement of $R_w$. We will show that  this polynomial is the Poincar\'e polynomial $P_v(q)$ of a permutation $v$.

Define the permutation matrix of $v$ as follows: let $w=a_1a_2\cdots a_k b_1b_2\cdots b_{n-k}$ be the decomposition promised by Theorem~\ref{thm:indecomp}. Let $\lambda=(\lambda_1,\ldots,\lambda_n)$ and $\mu=(\mu_1,\ldots)$. For $i=1,\ldots,k$,  let $v_i = \min(([n]\setminus [\mu_i])\setminus \{v_1,\ldots,v_{i-1}\})$ and for $j=1,\ldots,n-k$ let $v_{n+1-j} = \max([\lambda_{n-j}]\setminus \{v_{n-j+1},\ldots,v_n\})$. This defines a $0$-$1$ matrix with exactly one $1$ in every row; it follows from the proof of Theorem~\ref{thm:indecomp} that this matrix is in fact a permutation matrix (with $\{v_1,\ldots,v_k\}=[k]$ and $\{v_{k+1},\ldots,v_{n}\}= \{n-k+1,\ldots,n\}$). See Figure~\ref{fig:exSkewPoin} for an example of this construction. It is clear that $S_{\lambda/\mu(w)} = H_L(v)$. Also by Proposition~\ref{prop:nrksskew} we have that $q^{\binom{n}{2}-|\mu|}R_n^{\sf (NE)}(S_{\lambda/\mu}(w),q^{-1})=R_n^{\sf (SE)}(H_L(v),q)$.

By construction the prefix $v_1\cdots v_k$ avoids $132$ and the suffix $v_{k+1}\cdots v_{n}$ avoids $213$ and $v_j>v_i$ for $j\geq k+1$ and $i\leq k$.
It is easy to see that the set of permutations with such a decomposition is closed under containment of patterns, and does not contain any of the permutations $1324$, $24153$, $31524$, and $426153$. Therefore, every permutation in this set, and in particular $v$, avoids these four patterns. By Theorem~\ref{sjothm} it follows that $q^{|\mu|}R_n^{\sf (SE)}(H_L(v),q)=P_v(q)$. Thus
\begin{align*}
\mat_q(n, R_w,n)/(q-1)^n  &=  q^{n^2 - \inv(w) - n}R_n^{\sf (NE)}(S_{\lambda/\mu(w)},q^{-1})\\
&= q^{\binom{n}{2}+|\mu|-\inv(w)} R_n^{\sf (SE)}(H_L(v),q)\\
&=  q^{\binom{n}{2}-\inv(w)}(q-1)^n\cdot P_v(q),
\end{align*}  
as desired.
\end{proof}

\begin{example}
By Theorem~\ref{thm:indecomp}, the permutation $w = 21534$ is skew-vexillary.  After rearranging rows and columns (see Figure~\ref{fig:exSkewPoin}), the skew shape $S_{\lambda/\mu(w)}$ is $S_{55553/1}$. The associated $v$ is $21453$ and we have
\begin{align*}
\mat_q(5,R_{21534},5) &= q^{10-3}(q-1)^5P_{21453}(q) \\
&= q^{7}(q-1)^5(       	
q^{10} + 4  q^{9} + 9  q^{8} + 14  q^{7} + 15  q^{6} + 11 
q^{5} + 5  q^{4} + q^{3}
).
\end{align*}
\end{example}

\begin{remark}
Note that the result above does not hold for all Rothe diagrams.  There exist permutations $w$ for which there does not exist any permutation $v$ such that $\mat_q(n, R_w,r) = q^{\binom{n}{2}-\inv(w)}(q-1)^rP_v(q)$.  For example, take $w = 31524$ (see Figure~\ref{fig:exRothe} (iii) and Table~\ref{fig:4specpatt}). In this case
\[ 
\mat_q(5, R_{31524},5)=q^{6}(q-1)^5(q^{10}+4q^9+9q^8+12q^7+10q^6+5q^5+q^4).
\] 
One can show (either by computer search or by a direct argument about the possible structure of the inversions) that there is no permutation $v$ in $\mathfrak{S}_5$ such that $P_v(q)=q^{10}+4q^9+9q^8+12q^7+10q^6+5q^5+q^4$.
\end{remark}



We have shown that for a skew-vexillary permutations $w$, $\mat_q(n,R_w,n)/(q-1)^n$ is equal  (up to a power of $q$) to the Poincar\'e polynomial of some permutation $v$. Next we consider the problem of classifying permutations $w$ such that $\mat_q(n,R_w,n)/(q-1)^n$ is equal (up to a power of $q$) to the Poincar\'{e} polynomial  of the same permutation.

\subsection{Further  relationships between \texorpdfstring{$\mat_q(n,R_w,n)$}{mat(q;n,Rw,n)} and Poincar\'e polynomials} 

Computational evidence for $n\leq 7$ \cite{Code} suggests the following conjecture.
\begin{conjecture} \label{conj:PoinRothe}
Fix a permutation $w$ in $\mathfrak{S}_n$ and let $R_w$ be its Rothe diagram. We have that $\mat_q(n,R_w,n)/(q-1)^n$ is coefficient-wise less than or equal to $q^{\binom{n}{2}-\inv(w)}P_{w}(q)$. We have equality if and only if $w$ avoids the patterns $1324, 24153, 31524$, and $426153$.
\end{conjecture}

\begin{table}
{\small
\[
\begin{array}{c|c|c}
w&1324 & 24153 \text{ or } 31524 \\ \hline  
\frac{\mat_q(n,R_w,n)}{(q-1)^nq^{k}}&
{q}^{6}+3\,{q}^{5}+5\,{q}^{4}+5{q}^{3}+3\,{q}^{2}+{q}\phantom{+1}& {q}^{10}+4{q}^{9}+9{q}^{8}+12{q}^{7}+10{q}^{6}+5{q}^{5}
+{q}^{4}\\[2mm]
P_w(q)&
q^{6} + 3q^{5} + 5q^{4} + 6q^{3} + 4q^{2} + q\phantom{+1}&        	
q^{10} + 4q^{9} + 9q^{8} + 13q^{7} + 11q^{6} + 5q^{5} + q^{4}\\[2mm]
q^aR_n^{\sf (SE)}(H_L(w))& q^6+3q^5 + 5q^4 + 6q^3 + 5q^2 + 3q + 1  & q^{10} + 4q^9 + 9q^8 + 13q^7 + 12q^6+7q^5 + 2q^4 \\[3mm]
\hline \hline  
w& \multicolumn{2}{c}{426153}\\ \hline
\frac{\mat_q(n,R_w,n)}{(q-1)^nq^k}&  \multicolumn{2}{c}{{q}^{15}+5{q}^{14}+14{q}^{13}+24{q}^{12}+27{q}^{11}+19{q}^{
10}+7{q}^{9}+{q}^{8}}\\[2mm]
P_w(q)&\multicolumn{2}{c}{
 q^{15} + 5q^{14} + 14q^{13} + 25q^{12} + 28q^{11} + 19q^{10} + 7q^{9} + q^{8}}\\[2mm] 
q^aR_n^{\sf (SE)}(H_L(w))&\multicolumn{2}{c}{q^{15} + 5q^{14} +14q^{13} + 25q^{12} +29q^{11} + 21q^{10} + 8q^9 + q^8}
\end{array}
\]
}
\caption{For the four special patterns $w$ of Conjecture~\ref{conj:PoinRothe} we give $\mat_q(n,R_w,n)/((q-1)^nq^k)$ where $k=\binom{n}{2}-\inv(w)$, the Poincar\'e polynomials $P_w(q)$, and $q^aR_n^{\sf (SE)}(H_L(w),q)$ where $a$ is the size of the subtracted partition of the skew shape $H_L(w)$.} \label{fig:4specpatt}
\end{table}

\begin{remark} 
The patterns that appear in  Conjecture~\ref{conj:PoinRothe} and in Theorem~\ref{sjothm} are the same. Also, the reverses  $4231, 35142, 42513$, and $351624$ of these patterns have appeared in related contexts in a conjecture of Postnikov \cite{AP} proved by Hultman-Linusson-Shareshian-Sj\"ostrand \cite{HLSS}, and in work by Gasharov-Reiner \cite{GR2}. This suggests further interesting connections. These permutations were recently enumerated by  Albert-Brignall \cite{AB}. 
\end{remark}

The values of the three polynomials $\mat_q(n,R_w,n)/(q-1)^n$, $P_w(q)$, and $R^{\sf (SE)}_n(H_L(w),q)$ when $w$ is equal to the four patterns of Conjecture~\ref{conj:PoinRothe} are shown in Table~\ref{fig:4specpatt}. In these cases the three polynomials are all different. By Theorem~\ref{sjothm} and Theorem~\ref{NEthm}, Conjecture~\ref{conj:PoinRothe} is equivalent to the following:


\begin{conjecture}\label{conj:RookRothe}
Fix a permutation $w$ in $\mathfrak{S}_n$, let $R_w$ be its Rothe diagram and let $a_w = n^2 - \#H_L(w) - \inv(w)$. We have that $\mat_q(n,R_w,n)/(q-1)^n$ is coefficient-wise less than or equal to $q^{a_w}\mat_q(n,\overline{H_L(w)},n)/(q-1)^n$. We have equality if and only if $w$ avoids the patterns $1324$, $24153$, $31524$, and $426153$.
\end{conjecture}

This conjecture is not true for matrices of lower rank. For example, for $w=21\in \mathfrak{S}_2$ we have $\mat_q(2,R_{21},1)/(q-1)=2q+1$ and $\mat_q(2,\overline{H_L(21)},1)/(q-1)=2$. 

\begin{remark}
If Conjecture~\ref{conj:RookRothe} holds then by Theorem~\ref{NEthm} and \cite[Prop. 5.1]{LLMPSZ} it follows that whenever $w$ avoids the four patterns, the shapes $\overline{R_w}$ and $H_L(w)$ have the same number of placements of $n$ non-attacking rooks. This is not obvious since for such permutations the shapes are distinct even after permuting rows and columns. Moreover, computer experiments for $n\leq 7$ \cite{Code} suggest that the converse is also true, i.e., if $w$ contains any of the four patterns, the shapes have different numbers of rook placements. This apparent equivalence of necessary and sufficient conditions between the ``$q$ case'' and the ``$q=1$ case'' does not necessarily hold in similar settings (see \cite[Thm. 7]{OPY} and \cite[Thm. 3.4]{HLSS} for an example).  
\end{remark}

We end by giving a very preliminary step in proving these conjectures. 

\begin {proposition}\label{prop:numzeroes}
If $w$ is a $1324$-avoiding permutation then the complement of $H_L(w)$ has at least as many entries as the Rothe diagram $R_w$ of $w$.  
\end {proposition}

\begin{proof}
Let $w$ be a $1324$-avoiding permutation. We give a one-to-one map $\varphi$ between the entries of the Rothe diagram $R_w$ and the complement of the left hull $H_L(w)$.

Given an entry $(i,j)$ in $R_w$ we have two possibilities: either there is or there is not an entry  $(k,w_k)$ of $w$ such that $k<i$ and $w_k<j$ (i.e., an entry of $w$ NW of $(i,j)$). Let $A_w$ be the set of entries of $R_w$ of the first type and let $B_w$ be the set of entries of the second type.  If $(i,j)\in A_w$ then define $\varphi(i,j)=(i,j)$. If instead $(i,j)\in B_w$ then define $\varphi(i,j)=(w^{-1}_j,w_i)$. See Figure~\ref{fig:exnumzeroes} for an illustration of $\varphi$. We show that $\varphi$ is  well-defined and injective.

\begin{figure}
\begin{center}
\includegraphics[height=2.8cm]{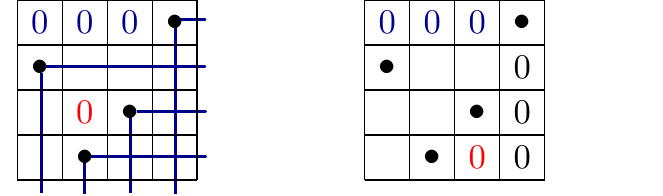}
\end{center}
\caption{Example of Proposition~\ref{prop:numzeroes}. For the permutation $w=4132$, the Rothe diagram is $R_w=\{(1,1),(1,2),(1,3),(3,2)\}$ and the left-hull $\overline{H_L(w)}=\{(1,1),(1,2),(1,3),(2,4),(3,4),(4,3),(4,4)\}$. The map $\varphi$ is given by $\textcolor{DarkBlue}{(1,j)}\mapsto \textcolor{DarkBlue}{(1,j)}$ for $j=1,2,3$ and $\textcolor{red}{(3,2)}\mapsto \textcolor{red}{(4,3)}$.}\label{fig:exnumzeroes}
\end{figure}

Choose $(i,j)$ in $R_w$. There are no entries of $w$ above $(i,j)$ in the same column or to its left in the same row. If in addition $(i,j)$ is in $A_w$ then by definition of the left hull the entry $(i,j)$ is not in $H_L(w)$. In this case $\varphi(i,j)=(i,j)\in \overline{H_L(w)}$ as desired. 

On the other hand, if $(i,j)$ is in $B_w$ then there is some entry $(k,w_k)$ of $w$ with $k<i$ and $w_k<j$. Since $w$ is $1324$-avoiding, there can be no entry $(\ell,w_{\ell})$ of $w$ such that $\ell\geq w^{-1}_j$ and $w_{\ell}\geq w_i$. Thus, $\varphi(i,j)=(w^{-1}_j,w_i) \in \overline{H_L(w)}$. This completes the proof that the map $\varphi$ is well-defined.

Finally, we  show that $\varphi$ is one-to-one. Since $\varphi$ is defined piecewise it is enough to show that $\varphi$ is one-to-one on $A_w$ and $B_w$ and that $\varphi(A_w)$ and $\varphi(B_w)$ are disjoint. The injectivity on $A_w$ is trivial. The injectivity on $B_w$ follows since $w$ is a permutation and so $(w^{-1}_j,w_i)$ uniquely defines $(i,j)$. Moreover, $\overline{H_L(w)}$ has two components; $\varphi(A_w)$ is the NW component while $\varphi(B_w)$ is contained in the SE component, so the images are disjoint. This completes the proof that $\varphi$ is one-to-one.
\end{proof}


\bibliography{biblio-Primes}{}
\bibliographystyle{plain}

\noindent Aaron J. Klein, 
Brookline High School.
Brookline, MA USA 02445.\\
{\tt ajmath62@gmail.com}

\medskip 

\noindent Joel B. Lewis, 
School of Mathematics, University of Minnesota, Minneapolis, MN 55455 USA.\\
{\tt jblewis@math.umn.edu}

\medskip

\noindent Alejandro H. Morales, 
Laboratoire de Combinatoire et d`informatique mathematique (LaCIM)
Universit\'e du Qu\'ebec \`a Montr\'eal,
Montr\'eal, QC, H3C 3P8 Canada.\\
{\tt ahmorales@lacim.ca}

\end{document}